\documentclass[11pt]{article}
 \usepackage[scale=.8]{geometry}
\usepackage{graphicx}
\usepackage{amssymb}
\usepackage{amsmath}
\usepackage{hyperref}
\usepackage{enumerate}
\usepackage{amsthm}
\usepackage{amsmath}
\usepackage{float}
\usepackage{mathtools}
 \usepackage{lineno}
\usepackage{bbm}
\usepackage{textcomp}
\usepackage{esint}
\usepackage{cleveref}
 \usepackage{extarrows}

\usepackage{xcolor}
  \definecolor{lanse}{RGB}{0,0,255} 
\definecolor{zise}{RGB}{112,48,160} 
 \definecolor{hongse}{RGB}{200,0,0} 
\usepackage{lineno}
\makeatletter
\renewenvironment{proof}[1][\proofname]{%
   \par\pushQED{\qed}\normalfont%
   \topsep6\p@\@plus6\p@\relax
   \trivlist\item[\hskip\labelsep\bfseries#1\@addpunct{.}]%
   \ignorespaces
}{%
   \popQED\endtrivlist\@endpefalse
}
\makeatother

\numberwithin{equation}{section}

\newtheorem{theorem}{Theorem}
\newtheorem{proposition}[theorem]{Proposition}
 \numberwithin{theorem}{section}
 \newtheorem{corollary}[theorem]{Corollary}
\newtheorem{lemma}[theorem]{Lemma}

\newtheorem{remark}[theorem]{Remark}

\makeatletter
\def\keywords{\xdef\@thefnmark{}\@footnotetext}
\makeatother

\renewcommand{\P}{\mathbb{P}}
\newcommand{\W}{\mathbb{W}}
\newcommand{\E}{\mathbb{E}}
\newcommand{\R}{\mathbb{R}}

\newcommand{\cH}{\mathbb{H}}

\newcommand{\cN}{\mathcal{N}}

\newcommand{\cF}{\mathcal F}

\newcommand{\cM}{\mathcal M}

\newcommand{\eps}{\varepsilon}

 \newcommand{\nn}{\nonumber}
 \newcommand{\no}{\noindent}

\newcommand{\Extra}[1]{{\color{blue}#1}}
\renewcommand{\Extra}[1]{}
\newcommand{\la}{\langle}
\newcommand{\ra}{\rangle}

\begin{document}


\author{
Zeteng Fan$^*$ \quad Jieliang Hong$^\dagger$ \quad Jie Xiong$^{\star}$
}
\title{Some quenched and annealed limit theorems of superprocesses in random environments}

\date{{\small  {\it  
$^{*}$$^{\dagger}$$^{\star}$Department of Mathematics, Southern University of Science and Technology, \\Shenzhen, China\\
  $^{*}$E-mail:  {\tt fanzet@hotmail.com} \\
   $^\dagger$E-mail:  {\tt hongjl@sustech.edu.cn} \\
  $^{\star}$E-mail:  {\tt xiongj@sustech.edu.cn} 
  }
  }
  }
 
 \maketitle
\begin{abstract}
 Let $X=(X_t, t\geq 0)$ be a superprocess in a random environment described by a Gaussian noise $W=\{W(t,x), t\geq 0, x\in \R^d\}$ white in time and colored in space with correlation kernel  $g(x,y)$. When $d\geq 3$, under the condition that the correlation function $g(x,y)$ is bounded above by some appropriate function $\bar{g}(x-y)$, we present the quenched and annealed Strong Law of Large Numbers and the Central Limit Theorems regarding the weighted occupation measure $\int_0^t X_s ds$ as $t\to \infty$.  
\end{abstract}

\section{Introduction}\label{s1}

The Dawson-Watanabe superprocess, also called the super-Brownian motion (SBM), is a measure-valued branching process that arises as a scaling limit of the critical branching particle system (see, e.g., Watanabe \cite{Wat68} and Dawson \cite{Daw75}).  The super-Brownian motion has been shown to appear in various models such as the voter model \cite{CDP},  the contact process \cite{DP99}, the lattice trees \cite{DS}, \cite{Hol08}, the oriented percolation \cite{HS03}, and more recently, the SIR epidemic process \cite{Hong23}. In the classical setting of the SBM, all the underlying particles in the system move and branch independently of all other particles. The law of the SBM can be uniquely determined by two ingredients: the spatial motion and the branching mechanism. In this paper, we study a variant of the classical SBM--the superprocess in random environments proposed by Mytnik \cite{M96}. The model appears as the scaling limits of the branching particle system where the branching probabilities depend on the random environment. We refer the reader to Mytnik \cite{M96} for a detailed description of the approximating particle system. The associated martingale problem is stated below in \eqref{eMP}.

The well-known work by Iscoe \cite{Isc86} and \cite{Isc86b} investigate the asymptotic behavior of the occupation measure of the super-Brownian motion on $\R^d$, which is defined as follows: If $(X_t, t\geq 0)$ is some measure-valued process on $\R^d$, then the occupation measure $Y_t$ is given by
\begin{align} \label{aae}
Y_t(\cdot):=\int_0^t X_s(\cdot) ds.
\end{align}
For any integrable function $\phi$ and measure $\mu$, we write
\[
\la \mu, \phi\ra=\mu(\phi)=\int \phi(x) \mu(dx).
\]

 Let $\lambda$ be the Lebesgue measure on $\R^d$. Denote by $P_\lambda$  the law of the super-Brownian motion starting from $X_0=\lambda$.
When $d\geq 3$, Iscoe \cite{Isc86b} proves that with $P_\lambda$-probability one,
\begin{align} \label{aae1}
 \lim_{t\to \infty} t^{-1} \la Y_t, \phi\ra=\la \lambda, \phi\ra, \quad \forall \phi\in C_c(\R^d).
\end{align}
Here, we use $C_c(\R^d)$ to denote the space of continuous functions with compact supports. The above can be seen as a Law of Large Number (LLN) type convergence theorem. Moreover, Iscoe establishes in \cite{Isc86} the Central Limit Theorem (CLT) type results: For any $\phi\in C_c(\R^d)$, there are some $b_t>0$ and $\sigma(\phi)>0$ such that as $t\to \infty$,
\begin{align} \label{aae2}
\frac{ \la Y_t, \phi\ra-t\la \lambda, \phi\ra }{b_t } \Rightarrow \sigma(\phi) Z,
\end{align}
where the random variable $Z\sim \cN(0,1)$ is standard normal and we use $\Rightarrow$ to denote the weak convergence of random variables. The choice of the normalizing term $b_t$ depends much on the spatial dimension:
\begin{align*}
b_t=
\begin{cases}
t^{3/4}, \quad & d=3,\\
\sqrt{t\log t}, \quad & d=4,\\
t^{1/2}, \quad & d\geq 5.
\end{cases}
\end{align*}
 
In the lower dimensions $d\leq 2$, the asymptotic behaviors of $Y_t$  turn out to be rather different: When $d=1$, Iscoe \cite[Theorem 4.3]{Isc86} implies that $\lim_{t\to \infty}\la Y_t, \phi\ra <\infty$ for all $\phi\in C_c(\R)$ a.s.; when $d=2$, Iscoe \cite[Theorem 2]{Isc86b} shows that
\begin{align} \label{aae3}
 t^{-1} \la Y_t, \phi\ra  \Rightarrow \xi \cdot \la \lambda, \phi\ra,  
\end{align}
 where $\xi$ is some nontrivial, strictly positive, and infinitely divisible random variable.\\

Given the above limit theorems of the super-Brownian motion, we will study in the current manuscript the case when $X_t$ is a superprocess in random environments. Since there are two different sources of randomness (i.e., the superprocess and the random environment), one would expect the limiting behavior for $X_t$ or $Y_t$ to rely on the strength of the environment aside from the dependence on the spatial dimension. The dimensional dependence comes from the fact that there is plenty of space for the particles to diffuse in high dimensions so that the particles would be less correlated than in the lower dimensions. Similarly, we expect that if the random environment is ``weakly'' correlated, there ought to be LLN and CLT for the superprocess in random environments in high dimensions. These heuristics lead to our main results below.\\

To formally state our results, we provide the necessary notations and definitions. 
Denote by  $C_b^k(\R^d)$  the space of all bounded continuous functions on $\R^d$ with bounded continuous derivatives up to order $k$. 
Let $X=(X_t, t\geq 0)$ be a superprocess in a random environment with covariance function $g(x,y)$ defined on some complete filtered probability space $(\Omega, \cF, \cF_t, \P)$ such that $X$ satisfies the following martingale problem: 
\begin{align} \label{eMP}
(MP)_{X_0}: \quad &\text{ For any }  \phi \in C_b^2(\R^d), \ M_t(\phi)=X_t(\phi)-X_0(\phi)-\int_0^t X_s(\frac{\Delta}{2}\phi) ds\nn\\
&\text{ is  a continuous $(\cF_t)$-martingale with } \\
&\langle M(\phi)\rangle_t=\int_0^t X_s(\phi^2)ds+\int_0^t ds \int_{\R^d} \int_{\R^d} g(u,v) \phi(u)\phi(v) X_s(du) X_s(dv).\nn
\end{align}
We refer the reader to Mytnik \cite{M96} and Mytnik-Xiong \cite{MX07} for more information on the above martingale problem. Throughout the rest of the paper, we use $(X_t, t\geq 0)$ to denote the unique solution to the above martingale problem and let $Y_t$ be defined by \eqref{aae}. We assume $X_0=\lambda$ following Iscoe's setting. Denote by $\P_\lambda$ the law of the solution $X$ to $(MP)_{\lambda}$.\\

To study the long-time behavior of the superprocess in random environments,  we introduce the conditional martingale problem from \cite{MX07}. Let $S_0$ be the linear span of the set of functions $\{g(x,\cdot): x\in \R^d\}$. Define an inner product on $S_0$ by
\[
\la g(x,\cdot), g(y,\cdot)\ra_{\cH}=g(x,y),
\]
where $\cH$ is the completion of $S_0$ with respect to the norm $\|\cdot\|_{\cH}$ induced by $\la \cdot, \cdot\ra_{\cH}$. Let $\{e_i: i\geq 1\}$ be a complete orthonormal basis of $\cH$. Let $\{B^i: i\geq 1\}$ be a collection of independent linear Brownian motions defined on some probability space $(\Omega, \cF, \P)$. Set
\begin{align}  \label{aee2.1}
 W(t,x):=\sum_{i=1}^\infty \la g(x,\cdot), e_i \ra_{\cH} B^i(t), \quad \forall t\geq 0, x\in \R^d.
\end{align}
Notice that
\begin{align*}  
 \sum_{i=1}^\infty \la g(x,\cdot), e_i \ra_{\cH}^2 =\la g(x,\cdot), g(x,\cdot) \ra_{\cH}=g(x,x)<\infty.
\end{align*}
Hence, $W=\{W(t,x): t\geq 0, x\in \R^d\}$ is a Gaussian noise white in time and colored in space with
\begin{align}  \label{aee2.2}
 \E\Big[ W(t,x) W(s,y)\Big]=g(x,y) \cdot (s\wedge t).
\end{align}

Define 
\begin{align*} 
\cF_t^W:=\sigma\{W(s,x): s\leq t, x\in \R^d\}=\sigma\{B^i(s): s\leq t, i\geq 1\},
\end{align*}
and set
\begin{align*} 
\cF_\infty^W:=  \bigvee_{t\geq 0} \cF_t^W.
\end{align*}
 For an $\cF_t^W$-predictable function $f(s,x)$, we define the stochastic integral of $f$ with respect to $W$ by
\begin{align}  
 \int_0^t f(s,x) W(ds, x):=\sum_{i=1}^\infty  \int_0^t \la f(s,x) g(x,\cdot), e_i \ra_{\cH} dB_s^i.
\end{align}
By using the linearity of the inner product, one may obtain that
\begin{align}  \label{ea4.49}
 \int_0^t \int_{\R^d} f(s,x) W(ds, x) dx=\sum_{i=1}^\infty  \int_0^t \Big\la \int f(s,x) g(x,\cdot) dx, e_i \Big\ra_{\cH} dB_s^i.
\end{align}
In particular, the above is an $\cF_t^W$-martingale with the quadratic variation given by
\begin{align*}  
& \sum_{i=1}^\infty  \int_0^t \Big\la \int f(s,x) g(x,\cdot) dx, e_i \Big\ra_{\cH}^2 ds\nn\\
 &=\int_0^t \Big\la \int f(s,x) g(x,\cdot) dx, \int f(s,y) g(y,\cdot) dy \Big\ra_{\cH}  ds\nn\\
 &=\int_0^t  \int  \int f(s,x)  f(s,y) g(x,y)  dx dy ds.
\end{align*}
Hence, we may rewrite \eqref{ea4.49} using the notation from the display above (2.2) of \cite{MX07}:
\begin{align}  \label{ea4.41}
 \int_0^t \int_{\R^d} f(s,x) W(ds, x) dx= \int_0^t \Big\la \int f(s,x) g(x,\cdot) dx, d\W_s \Big\ra_{\cH},
\end{align}
where $\W$ is some $\cH$-cylindrical Brownian motion.\\

By rewriting the last term of (2.2) in \cite{MX07} by using \eqref{ea4.41}, we conclude from Lemma 2.3 of \cite{MX07} that if $X$ satisfies the following conditional martingale problem:
\begin{align} \label{eCMP}
(CMP)_{X_0}: \quad &\text{ For any }  \phi \in C_b^2(\R^d), \ N_t(\phi)=X_t(\phi)-X_0(\phi)-\int_0^t X_s(\frac{\Delta}{2}\phi) ds\\
&-\int_0^t \int \phi(x)W(ds, x) X_s(dx) \text{ is  a continuous $(\cF_t \vee \cF_{\infty}^W)$-martingale}\nn  \\
&\text{  with } \langle N(\phi)\rangle_t=\int_0^t X_s(\phi^2)ds,\nn
\end{align}
then $X$ is a solution to $(MP)_{X_0}$.\\

We collect all the known results for $X_t$ and $Y_t$. It is always assumed that $$\|g\|_\infty=\sup_{x,y\in \R^d} |g(x,y)|<\infty.$$ Fix any $\phi \in C_c(\R^d)$. Theorems 1.1-1.3 of \cite{MX07} show that when $d=1$, there is a finite random time $\tau>0$ such that with $\P_\lambda$-probability one,  
\begin{align} \label{aee1.1}
\la X_t, \phi \ra=0, \quad \forall t\geq \tau \quad  \Longrightarrow \quad \limsup_{t\to \infty} \  \la Y_t, \phi \ra<\infty.
\end{align}
When $d=2$, we have
\begin{align} \label{aee1.2}
\lim_{t\to\infty} \la X_t, \phi \ra=0 \text{ in probability.}
\end{align}
For any $d\geq 2$, if there are some constant $0<c_1<c_2$ and $0<\alpha<2$ such that
 \begin{align}\label{aae0.0}
 c_1 (|x-y|^{-\alpha} \wedge 1)\leq g(x,y) \leq c_2, \quad \forall x,y\in\R^d,
\end{align}
then \eqref{aee1.1} holds; if $\alpha=2$, then \eqref{aee1.2} holds.\\

Turning to the case when $d\geq 3$ and $\alpha>2$, in an unpublished note, Chen-Ren-Zhao \cite{CRZ15} proves that if there is some $\eps>0$ small such that
 \begin{align}\label{e0.0}
g(x,y) \leq \eps (|x-y|^{-\alpha} \wedge 1), \quad \forall x,y\in\R^d,
\end{align}
then $X_t$ converges weakly to some random measure $X_\infty$ with $\E_\lambda \la X_\infty, \phi \ra=\la \lambda, \phi \ra$. On the other hand,  if there is some constant $C>0$ large such that $g(x,y)=g(x-y)$ with $g\in C^2(\R^d)$ and $g(0) \geq C$, then \eqref{aee1.1} holds.\\

From now on, we set $d\geq 3$ and $\alpha>2$. We will only consider the case when the random environment is ``weak'' enough in the sense that \eqref{e0.0} holds. The following lemma justifies our choice of \eqref{e0.0}.
Let $B_t, \tilde{B}_t$ be independent $d$-dimensional Brownian motion starting respectively from $x,y\in \R^d$ under $\Pi_{(x,y)}$.

\begin{lemma}\label{l1.1}
For any $\alpha>2$, $q>0$ and $d\geq 3$, there is some constant $c=c(\alpha, d, q)>0$ such that if $\eps\leq c$, then
\begin{align} \label{e1}
\sup_{x,y\in \R^d} \Pi_{(x,y)} \Big(e^{q\int_0^\infty   g(B_s, \widetilde{B}_s) ds} \Big)\leq 2
\end{align}
holds for all function $g(x,y)$ such that $g(x,y) \leq \eps (|x-y|^{-\alpha} \wedge 1)$.
\end{lemma}

The constant $2$ on the right-hand side of \eqref{e1} could be replaced by any constant larger than $1$; we choose $2$ for simplicity.
The proof of Lemma \ref{l1.1} deviates from our discussions of the superprocess, so we defer it to Appendix \ref{a1}.
In what follows, we will pick $q>0$ depending only on $d$ and $\alpha$. Hence, we may choose $\eps=\eps(d,\alpha)>0$ to be small such that \eqref{e1} holds.  \\

Now, we state our main results. Denote by $\P_{\lambda}^W$ the law of $\P_{\lambda}$ conditioning on $\cF_\infty^W$.

\begin{theorem}[Quenched LLN] \label{t1}
Let $d\geq 3$ and $\alpha>2$. There is some constant $c=c(\alpha, d)>0$ such that if $\eps\leq c$ and the covariance function $g$ satisfies \eqref{e0.0}, then for $\P$-a.s. environment $W$, we have $\P_{\lambda}^W$-a.s. that
\begin{align*} 
\lim_{T\to \infty} {T}^{-1} Y_T(\phi)   =\la \lambda, \phi\ra
\end{align*}
holds for all $\phi \in C_c(\R^d)$.
\end{theorem}
The proof of Theorem \ref{t1} will be given in Section \ref{as2}.
 Given the above, the following annealed version is immediate.

\begin{corollary}[Annealed LLN]\label{c1.1}
Let $d\geq 3$. For any $\alpha>2$, with $\P_{\lambda}$-probability one, we have
\begin{align*} 
\lim_{T\to \infty} {T}^{-1} Y_T(\phi)   =\la \lambda, \phi\ra
\end{align*}
holds for all $\phi \in C_c(\R^d)$.
\end{corollary}
\begin{proof}
For any $\phi \in C_c(\R^d)$, we get
\begin{align*} 
&\P_{\lambda}\Big(\lim_{T\to \infty} {T}^{-1} Y_T(\phi) \neq \la \lambda, \phi\ra\Big)\\
=&\P\Big[\P_{\lambda}^W\Big(\lim_{T\to \infty} {T}^{-1} Y_T(\phi) \neq \la \lambda, \phi\ra\Big)\Big]=0.
\end{align*}
The proof is complete by considering a countable determining class of $C_c(\R^d)$.
\end{proof}

\begin{remark}
The conclusions in Theorem \ref{t1} and Corollary \ref{c1.1} indeed imply the convergence of the measures ${T}^{-1} Y_T$ to $\lambda$ as $T\to \infty$ in the vague topology.
\end{remark}

Now that we have established the strong law of large numbers regarding $Y_T$, we will prove the associated central limit theorems. 

Let $\cM(\R^d)$ be the set of Radon measures on $\R^d$.
The following conditional Laplace transform is an easy consequence of \cite[Theorem 2.15]{MX07}:  For any $\mu \in \cM(\R^d)$ and any  $\phi, f \in C_c^+(\R^d)$, we have
\begin{align} 
 \E_\mu^W (e^{-X_t(f)- Y_t(\phi)  })= e^{-\la \mu, U^{f, \phi}(t,\cdot)\ra},
\end{align}
where $\E_\mu^{W}$ is the conditional expectation of $\E_\mu$ given ${W}$ and $U^{f, \phi}\geq 0$ is the solution to the following SPDE: 
\begin{align*} 
 \frac{\partial}{\partial t}U^{f, \phi}  (t,x)&= \phi(x)+  \frac{\Delta}{2} U^{\phi}(t,x)  - \frac{1}{2}  U^{\phi}(t,x)^2 +U^{\phi}(t,x)  \dot{W}(t,x), \\
U^{f, \phi}  (0,x)&=f(x).
\end{align*}
The derivation of the above from \cite[Theorem 2.15]{MX07} follows in a similar way to that of  Theorem 3.1 in \cite{Isc86} by using the Markov property of $X$. We also refer the reader to the proof of \cite[Theorem 2.18]{MX07} for a similar application. Now we set $U^{\phi}  (t,x)=U^{0, \phi}  (t,x)$ such that
\begin{align}\label{ed1.1}
 \E_\mu^W (e^{- Y_t(\phi)  })= e^{-\la \mu, U^{\phi}(t,\cdot)\ra},
\end{align}
where $U^{ \phi}\geq 0$ is the solution to the following SPDE: 
\begin{align}\label{ed10.3}
 U^{\phi}(t,x)=t\phi(x)+\int_0^t \frac{\Delta}{2} U^{\phi}(s,x) ds- \frac{1}{2} \int_0^t U^{\phi}(s,x)^2 ds  +\int_0^t U^{\phi}(s,x)  {W}(ds,x).
\end{align}

Set
\begin{align}\label{ae10.21}
p_t^x(y)=p_t(x,y)= \frac{1}{(2\pi t)^{d/2}} e^{-|y-x|^2/(2t)}
\end{align}
to be the transition density of the $d$-dimensional Brownian motion.
For any function $f$, define 
\begin{align}\label{ae10.99}
P_t f(x)=\int_{\R^d} p_t(x,y) f(y) dy, \quad \text{ and } \quad Q_t f(x)=\int_0^t P_s f(x) ds.
\end{align}
We may rewrite \eqref{ed10.3} as
\begin{align}\label{ea10.3}
U^{\phi}(t,x)=Q_t \phi(x)&- \frac{1}{2}\int_0^t ds \int p_{t-s}(x,y)  U^{\phi}(s,y)^2  dy\nn\\
&+\int_0^t ds \int p_{t-s}(x,y) U^{\phi}(s,y) \dot{W}(s,y) dy.
\end{align}
Define $V_1^{\phi}(t,x)$ by
\begin{align*} 
V_1^{\phi}(t, x)= \frac{\partial}{\partial \theta} U^{\theta \phi}(t,x)|_{\theta=0}.
\end{align*}
By differentiating \eqref{ed10.3} or \eqref{ea10.3} with respect to $\theta$ and noticing that $U^{0 \phi}(t,x)\equiv 0$ , we conclude that $V_1^{\phi}$ is the solution to the following SPDE: 
\begin{align}\label{ae10.33}
V_1^{\phi}(t,x)=t\phi(x)+\int_0^t \frac{\Delta}{2} V_1^{\phi}(s,x) ds   +\int_0^t V_1^{\phi}(s,x)  {W}(ds,x),
\end{align}
or equivalently,
\begin{align}\label{ae1.33}
V_1^{\phi}(t,x)=Q_t \phi(x)+\int_0^t ds \int p_{t-s}(x,y) V_1^{\phi}(s,y) \dot{W}(s,y) dy.
\end{align}
Replace $\phi$ in \eqref{ed1.1} by $\theta \phi$, differentiate $\theta$ and then let $\theta=0$ to obtain that 
\begin{align} \label{ea6.2}
 \E_\mu^W \Big( Y_t(\phi)   \Big)= \la \mu, V_1^{\phi}(t)\ra,
\end{align}
and hence
\begin{align*} 
 \E_\mu \Big( Y_t(\phi)   \Big)= \E \Big[ \E_\mu^W \Big( Y_t(\phi)   \Big)   \Big]= \la \mu, Q_t \phi(\cdot)\ra.
\end{align*}

Let $Z$ be a standard normal random variable defined on the same probability space $(\Omega, \cF, \P)$ such that $Z$ is independent of $W$. Denote by $\P^W$  the law of $\P$ conditional on $\cF_\infty^W$.

\begin{theorem}[Quenched CLT] \label{t2}
 Let $d\geq 5$ and $\alpha>2$. For any $\phi \in C_c^+(\R^d)$ and any sequence $T_n \to \infty$, there is a subsequence $T_{n_k} \to \infty$ such that for $\P$-a.s. environment $W$,  we have
\begin{align} 
\P^W\Big[T_{n_k}^{-1/2} \Big(Y_{T_{n_k}}(\phi)   -\la \lambda, V_1^{\phi}({T_{n_k}})\ra\Big) \in \cdot  \Big] \Rightarrow \P^W\Big(\sigma(W, \phi) Z \in \cdot \Big),
\end{align}
where $\sigma(W, \phi) \in (0,\infty)$ is defined by
\begin{align} \label{aee3.1}
 \sigma(W, \phi)^2:=\lim_{t\to \infty} \int \Big(V_1^\phi(t,x) \Big)^2 dx.
\end{align}
 
\end{theorem}

The proof of Theorem \ref{t2} will be given in Section \ref{as2}. Again, we have the following annealed version given the above-quenched result.

\begin{corollary}[Annealed CLT]\label{c2.2}
Let $d\geq 5$ and $\alpha>2$. For any $\phi \in C_c^+(\R^d)$, we have 
\begin{align*} 
\P \Big[T^{-1/2} \Big(Y_T(\phi)   -\la \lambda, V_1^{\phi}({T})\ra\Big) \in \cdot  \Big] \Rightarrow \P\Big(\sigma(W, \phi) Z \in \cdot \Big)
\end{align*}
as $T\to \infty$.
\end{corollary}

\begin{proof}
Fix any $\phi \in C_c^+(\R^d)$. By Theorem \ref{t2}, for any sequence $T_n\to \infty$, there is a further subsequence $T_{n_k} \to \infty$ such that for $\P$-a.s. environment $W$, we get for all $a\in \R$,
\begin{align*} 
&\lim_{T_{n_k}\to \infty} \P^W\Big[T_{n_k}^{-1/2} \Big(Y_{T_{n_k}}(\phi)   - T_{n_k}\la \lambda, \phi\ra  \Big)    \leq a \Big] =   \P^W \Big[\sigma(W, \phi) Z   \leq a \Big].
\end{align*}
It follows that
\begin{align*} 
&\lim_{T_{n_k}\to \infty} \P \Big[T_{n_k}^{-1/2} \Big(Y_{T_{n_k}}(\phi)   - T_{n_k}\la \lambda, \phi\ra  \Big)    \leq a \Big]  \\
=&\lim_{T_{n_k}\to \infty}\P\Big[\P^W\Big(T_{n_k}^{-1/2} \Big(Y_{T_{n_k}}(\phi)   - T_{n_k}\la \lambda, \phi\ra  \Big)    \leq a\Big) \Big]\\
= & \P\Big[\P_{\lambda}^W\Big(\sigma(W, \phi) Z   \leq a\Big) \Big]=\P\Big(\sigma(W, \phi) Z   \leq a\Big).
\end{align*}
Since $T_n$ is arbitrary, we conclude that
\begin{align*} 
&\lim_{T\to \infty} \P\Big[T^{-1/2} \Big(Y_{T}(\phi)   - T\la \lambda, \phi\ra  \Big)    \leq a \Big]  =\P\Big(\sigma(W, \phi) Z   \leq a\Big),
\end{align*}
as required.
\end{proof}
 
It would be desirable if one could replace $\la \lambda, V_1^{\phi}({T})\ra$ in Theorem \ref{t2} and Corollary \ref{c2.2} by $T\la \lambda,  \phi\ra$ suggested by Theorem \ref{t1} and Corollary \ref{c1.1}. However, the problem is that the limiting distribution may no longer be a centered Gaussian due to the random environment. We have the following result.

\begin{proposition}\label{p1.1}
Let $d\geq 5$ and $\alpha>4$. For any $\phi \in C_c^+(\R^d)$, as $T\to \infty$ we have 
\begin{align*} 
\P \Big[T^{-1/2} \Big(\la \lambda, V_1^{\phi}({T})\ra   -T\la \lambda,  \phi\ra\Big) \in \cdot  \Big] \Rightarrow \P\Big(B_{\xi(W,\phi)} \in \cdot \Big),
\end{align*}
where $(B_t, t\geq 0)$ is a linear Brownian motion and $\xi(W,\phi)\in (0,\infty)$ is given by
\begin{align} 
 \xi(W,\phi):= \lim_{t\to \infty}  \int \int  V_1^\phi(t,x) V_1^\phi(t,y) g(x,y) dx dy.
\end{align}
\end{proposition}

The proof of Proposition \ref{p1.1} will be given in Section \ref{as7}.\\

\no {\bf Conjectures and open problems}:\\

\no (1) For the central limit theorem in $d=3,4$, we do conjecture that both Theorem \ref{t2} and Corollary \ref{c2.2} hold if we replace $T^{-1/2}$ by $T^{-3/4}$ in $d=3$ and $T^{-1/2} (\log T)^{-1}$ in $d=4$.  \\

\no (2) The current manuscript proves the LLN and CLT when the random environment satisfies \eqref{e0.0}. Together with Chen-Ren-Zhao's results (see \cite{CRZ15}), we may summarize that when $\alpha>2$, there exist two constants $0<c_1\leq c_2$ such that if $g(x,y)\leq c_1 (|x-y|^{-\alpha} \wedge 1)$, then the above four results follow; if $g(x,y)\geq c_2 (|x-y|^{-\alpha} \wedge 1)$, then there are no such limit theorems in view of \eqref{aee1.1}. It is not clear whether $c_1=c_2$ or not. Will there be any other limiting theorems? We leave it as an open problem.\\

\no  (3) Assuming that $g$ satisfies \eqref{aae0.0} with $d=2$ and $\alpha\geq 2$ or $d\geq 3$ and $\alpha=2$, we have \eqref{aee1.2} holds. However, we do not know the limiting behavior of $Y_t$ as $t\to \infty$. There might exist some nontrivial limits of $t^{-1} Y_t$ similar to the super-Brownian motion case (see \eqref{aae3}). For instance, in $d=2$, if we set 
one needs to study the scaled noise
\begin{align} 
 w_T(t,x)=W(Tt, \sqrt{T}x),
 \end{align}
 where $W$ is the colored noise as in \eqref{aee2.1} and \eqref{aee2.2}.
 Assume for simplicity that
  \begin{align*} 
 g(x,y) =c (|x-y|^{-\alpha} \wedge 1).
\end{align*}
It follows that if $\alpha=2$, then for any $x\neq y$, we have
  \begin{align} 
\lim_{T\to \infty}\E[w_T(t,x) w_T(s,y)]=\lim_{T\to \infty} Tg(\sqrt{T} x, \sqrt{T} y) (t\wedge s)=c|x-y|^{-2} (t\wedge s);
 \end{align}
if $\alpha>2$, then
   \begin{align} 
\lim_{T\to \infty}\E[w_T(t,x) w_T(s,y)]=0, \quad \forall x\neq y.
 \end{align}
 We hope to return to this problem in future work.\\

  \no {\bf Organization of the paper}.   In Section \ref{as2}, we will give the proofs of the main results Theorem \ref{t1} and Theorem \ref{t2} assuming some moment bounds. Section \ref{as3} states all the moment formulas regarding $Y_T$ using the conditional Laplace transform.  Section \ref{as4} presents some preliminary estimates for the second moment. In Section \ref{as5}, we combine the moment formulas from Section \ref{as3} and the moment estimates from Section \ref{as4} to complete the proof of Theorem \ref{t1}. Section \ref{as6} establishes the convergence of the conditional Laplace transform and thus finishes the proof of Theorem \ref{t2}. In Section \ref{as7}, we use the Dubins-Schwartz Theorem to prove Proposition \ref{p1.1}.

\section*{Acknowledgement}

 Jie Xiong is supported by China's National Key R\&D Program (No. 2022YFA1006102).  Jieliang Hong is supported by the Startup Foundation of Shenzhen (No. Y01286145).

\section{Proofs of the main theorems}\label{as2}

In this section, we present the proofs of Theorem \ref{t1} and Theorem \ref{t2} by assuming some intermediate results.

 \subsection{Proof of Theorem \ref{t1}}
 
 We claim that the proof of Theorem \ref{t1} can be reduced to proving the following second moment bounds on $\la Y_T,  \phi \ra$.
\begin{lemma}\label{al2.1}
Let $d\geq 3$ and $\alpha>2$. There is some $\delta=\delta(d,\alpha)\in (0,1)$ such that for any $\phi\in C_c^+(\R^d)$, we have
\begin{align*}
\E_{\lambda}\Big[\Big(\la Y_T,  \phi \ra-T\la {\lambda}, \phi\ra\Big)^2\Big]\leq CT^{2-\delta}, \quad \forall T>0,
\end{align*}
where $C>0$ is some constant depending only on $d, \alpha, \phi$.
\end{lemma}

The proof of Lemma \ref{al2.1} will be done in Section \ref{as5}. We can prove Theorem \ref{t1} assuming the above result.

\begin{proof}[Proof of Theorem \ref{t1} assuming Lemma \ref{al2.1}]
 
Fix any $r>1$. For any $K\geq 1$, by using Chebyshev's inequality, we get  \begin{align*}
&\E\Big[\sum_{n=1}^\infty\P_{\lambda}^W\Big(\Big| r^{-n} Y_{r^n}(\phi)-\la {\lambda}, \phi\ra\Big|>K^{-1}\Big)\Big] \\
&\leq  K^{2}  \E\Big[ \sum_{n=1}^\infty  r^{-2n} \E_{\lambda}^W \Big(Y_{r^n}(\phi)-r^n \la {\lambda}, \phi\ra\Big)^2\Big]\\
&=  K^{2} \sum_{n=1}^\infty  r^{-2n}\E_{\lambda} \Big[ \Big(Y_{r^n}(\phi)-r^n \la {\lambda}, \phi\ra\Big)^2\Big] \leq CK^{2} \sum_{n=1}^\infty r^{-\delta n} <\infty,
\end{align*}
where the second inequality follows by Lemma \ref{al2.1}.
Therefore for $\P$-a.e. $W$, we get
\begin{align}\label{ea5.2}
\sum_{n=1}^\infty\P_{\lambda}^W\Big(\Big| r^{-n} Y_{r^n}(\phi)-\la {\lambda}, \phi\ra\Big|>K^{-1}\Big)<\infty.
\end{align}
Fix $\omega$ outside a $\P$-null set such that for each $W=W(\omega)$, we have \eqref{ea5.2} holds for all $K\geq 1$. By Borel-Cantelli's lemma, we conclude that 
\begin{align}\label{ea5.7}
\lim_{n\to\infty}  \frac{Y_{r^n}(\phi)}{r^{n}}=\la {\lambda}, \phi\ra, \quad \P_{\lambda}^W\text{-a.s.}
\end{align}
For all $t>0$ large, there is some $n\geq 1$ such that $r^{n}\leq t<r^{n+1}$. Hence by the monotonicity of $t\mapsto Y_t(\phi)$, we get
\begin{align*}
 \frac{Y_{r^{n}}(\phi)}{r^{n+1}} \leq \frac{Y_{t}(\phi)}{t}\leq  \frac{Y_{r^{n+1}}(\phi)}{r^{n}}.
\end{align*}
Together with \eqref{ea5.7}, we conclude that with $\P_{\lambda}^W$-probability one,
\begin{align*}
r^{-1} \la {\lambda}, \phi\ra \leq  \liminf_{t\to\infty}  \frac{Y_{t}(\phi)}{t}\leq \limsup_{t\to\infty}  \frac{Y_{t}(\phi)}{t}\leq r \la {\lambda}, \phi\ra.
\end{align*}
Let $r \downarrow 1$ to get the desired results.
\end{proof}

 \subsection{Proof of Theorem \ref{t2}}
 
We proceed to prove the Central Limit Theorem as in Theorem \ref{t2}.  Following the proof of Iscoe \cite[Theorem 5.4]{Isc86},   we only need to show that for any $\phi \in C_c^+(\R^d)$ and any sequence $T_n\to \infty$, there is a subsequence $T_{n_k} \to \infty$ such that for $\P$-a.s. environment $W$,  we have
\begin{align}\label{ae9.31}
 \lim_{T_{n_k} \to \infty}\E^W& \Big[\exp\Big(-\theta T_{n_k}^{-1/2} \Big(\la Y_{T_{n_k}}, \phi \ra   -\la {\lambda}, V_1^{\phi}({T_{n_k}})\ra \Big)\Big) \Big]
  =  e^{\frac{1}{2} \theta^2 \sigma(W, \phi)^2}, \quad \forall \theta \geq 0.  
\end{align}

Take $\mu=\delta_x$ in \eqref{ea6.2} to see that
\begin{align}\label{cea6.6}
 V_1^\phi(t, x)=\E_{\delta_x}^W \Big(\la Y_t, \phi\ra\Big).
\end{align}
By replacing $\phi$ by $\theta \phi$ for any $\theta\geq 0$ in the above, we get
\[
V_1^{\theta \phi}(t,x)= \theta V_1^{\phi}(t,x).
\]
It follows that (recall $\sigma(W,  \phi)$ from \eqref{aee3.1})
\begin{align} \label{ea6.12}
\sigma(W, \theta\phi)= \theta \sigma(W,  \phi), \quad \P\text{-a.s.}
\end{align}
In view of the above, the proof of \eqref{ae9.31} is reduced to  showing that
\begin{align}\label{cea6.7}
 \lim_{T_{n_k} \to \infty}\E^W& \Big[\exp\Big(- T_{n_k}^{-1/2} \Big(\la Y_{T_{n_k}}, \phi \ra   -\la {\lambda}, V_1^{\phi}({T_{n_k}})\ra \Big)\Big) \Big]
  =  e^{\frac{1}{2}  \sigma(W, \phi)^2}.  
\end{align}

 Fix any   $\phi \in C_c^+(\R^d)$. For each $T>0$, define
 \begin{align*} 
 Z_T^\phi :=& T^{-1/2} \Big(Y_T(\phi)   -\la {\lambda}, V_1^{\phi}(T)\ra\Big)\nn\\
 =& Y_T(T^{-1/2}  \phi)  -\la {\lambda}, T^{-1/2}  V_1^{\phi}(T)\ra.
\end{align*}
 It follows from the conditional Laplace transform \eqref{ed1.1} that
\begin{align} \label{ea7.2} 
 \E^W (e^{- Z_T^\phi})&= \E_\mu^W (e^{- Y_t(T^{-1/2}  \phi)  }) \cdot e^{\la {\lambda}, T^{-1/2}  V_1^{\phi}(T)\ra} \nn\\
 &=e^{-\la \lambda, U^{T^{-1/2}\phi}(T) \ra+\la {\lambda}, T^{-1/2}  V_1^{\phi}(T)\ra}= e^{ T^{-1/2}\la \lambda, v^{\phi}_T(T) \ra},
\end{align}
where we set
\begin{align} \label{ea6.1}
 v_T^\phi(t,x):=  V_1^{\phi}(t,x)-u_T^\phi(t, x)  \quad \text{ and } \quad u_T^\phi(t, x):=T^{1/2}   U^{T^{-1/2}\phi}(t,x).
 \end{align}
 
 It suffices to find the limit
 \begin{align*}
 \lim_{T\to \infty}T^{-1/2}\la \lambda, v^{\phi}_T(T) \ra.  
\end{align*}
In view of \eqref{ed10.3}, we obtain that
\begin{align}\label{ae2.2}
 u_T^\phi(t, x)= t \phi(x)&+\int_0^t \frac{\Delta}{2} u_T^\phi(s,x) ds\\
&-\frac{1}{2} T^{-1/2} \int_0^t  u_T^\phi(s,x)^2 ds  + \int_0^t  u_T^\phi(s,x) \dot{W}(s,x) ds.\nn
\end{align}
 By comparing the above with \eqref{ae10.33}, one may use the comparison principle for SPDEs (see, e.g., Theorem 2.26 of Pardoux \cite{Par}) to conclude that for any $T> 0$, with probability one, we have
\begin{align}\label{ae7.31}
  0\leq  v_T^\phi(t,x)= V_1^{\phi}(t,x)- u_T^\phi(t, x) \leq V_1^{\phi}(t,x), \quad \forall t\geq 0, x\in \R^d,
\end{align}

Combine \eqref{ae10.33} and \eqref{ae2.2} to see that
\begin{align*} 
v_T^\phi(t,x)= & \int_0^t \frac{\Delta}{2} v_T^\phi(s,x) ds+\frac{1}{2T^{1/2}}  \int_0^t  u_T^\phi(s,x)^2 ds  + \int_0^t  v_T^\phi(s,x) \dot{W}(s,x) ds. 
\end{align*}
Rewrite the above as
\begin{align}\label{ae7.34}
v_T^\phi(t,x)=\frac{1}{2T^{1/2}} & \int_0^t ds \int p_{t-s}(x,y) u_T^\phi(s,y)^2 dy\nn\\
 &+\int_0^t ds \int p_{t-s}(x,y) v_T^\phi(s,y) \dot{W}(s,y) dy.
 \end{align}
By integrating $x$ over  $\R^d$ on both sides above, we get
\begin{align}\label{ae9.32}
T^{-1/2}\la \lambda, v^{\phi}_T(T) \ra=\frac{1}{2T}  &\int_0^T dt \int  u_T^\phi(s,y)^2 dx+T^{-1/2}\int_0^T ds \int  v_T^\phi(t, x) \dot{W}(t, x) dx,
 \end{align}
where we have applied Fubini's theorem as well as the stochastic Fubini's theorem (see, e.g., Theorem 4.33 of \cite{Da14}).\\

 Define
 \begin{align}\label{ae7.32}
I_1(T, \phi):=&\frac{1}{2}  \frac{1}{T} \int_0^T dt \int  V_1^\phi(t, x)^2 dx,\\
I_2(T, \phi):=&\frac{1}{2}  \frac{1}{T}    \int_0^T dt \int  \Big(V_1^\phi(t, x)^2-u_T^\phi(t,x)^2\Big) dx,\nn\\
I_3(T, \phi):=&T^{-1/2}\int_0^T ds \int  v_T^\phi(t, x) \dot{W}(t, x) dx.\nn
\end{align}
Then \eqref{ae9.32} becomes
\begin{align}\label{ea7.1}
T^{-1/2}\la \lambda, v^{\phi}_T(T) \ra= I_1(T, \phi)-I_2(T, \phi)+I_3(T, \phi)
 \end{align}

Recall from \eqref{cea6.6} that  
\begin{align}\label{ea6.6}
 V_1^\phi(t, x)=\E_{\delta_x}^W [\la Y_t, \phi\ra] \quad \text{ is increasing in $t\geq 0$,}
\end{align}
The monotonicity implies that 
\begin{align*} 
 \lim_{t\to \infty} \int  V_1^\phi(t, x)^2 dx   \quad \text{ exists a.s.}
\end{align*}
To get the a.s. finiteness, we need the following moment estimates. The proof is deferred to Section \ref{as4}.

\begin{lemma}\label{al4.1}
Let $d\geq 5$ and $\alpha>2$. For any $\phi \in C_c^+(\R^d)$, we have
\begin{align*}
  \lim_{t\to \infty} \int  \E[V_1^\phi(t, x)^2] dx<\infty.
\end{align*}
\end{lemma}

Using the above with Fatou's lemma, we get
\begin{align*}
 \E\Big[\lim_{t\to \infty} \int  V_1^\phi(t, x)^2 dx\Big] \leq \liminf_{t\to \infty} \int  \E[V_1^\phi(t, x)^2] dx<\infty.
\end{align*}
Hence with $\P$-probability one, 
\begin{align}\label{ea6.7}
\sigma(W,\phi):= \Big(\lim_{t\to \infty} \int  V_1^\phi(t, x)^2 dx \Big)^{1/2}>0 \quad \text{ exists and is finite.}
\end{align}
It follows that
\begin{align}\label{ea6.3}
  \lim_{T\to \infty} I_1(T, \phi)=\frac{1}{2}   \sigma(W,\phi)^2, \quad \P\text{-a.s.}
\end{align}

 We will prove the following lemma in Section \ref{as6} to deal with the other two terms.
 
\begin{lemma}\label{al4.2}
Let $d\geq 5$ and $\alpha>2$. For any $\phi \in C_c^+(\R^d)$, we have
\begin{align}\label{ea6.4}
  \lim_{T\to \infty} \E\Big(|I_2(T, \phi)|\Big)=0,
\end{align}
and 
\begin{align}\label{ea6.5}
  \lim_{T\to \infty} \E[I_3(T, \phi)^2]=0.
\end{align}
\end{lemma}
 
 We are ready to finish the proof of Theorem \ref{t2}.
 
\begin{proof}[Proof of Theorem \ref{t2}  assuming Lemmas \ref{al4.1} and \ref{al4.2}]
We conclude from \eqref{ea7.1}, \eqref{ea6.3}, \eqref{ea6.4} and \eqref{ea6.5} that for any sequence $T_n \to \infty$, there is a subsequence $T_{n_k}\to \infty$ such that
\begin{align*} 
T^{-1/2}\la \lambda, v^{\phi}_T(T) \ra= I_1({T_{n_k}}, \phi)- I_2({T_{n_k}}, \phi)+I_3({T_{n_k}}, \phi) \to \frac{1}{2}   \sigma(W,\phi)^2, \quad \P\text{-a.s.}
 \end{align*}
 Now recall \eqref{ea7.2} to see that
\begin{align*}  
   \lim_{{T_{n_k}}\to \infty} \E^W (e^{- Z_{T_{n_k}}^\phi})= e^{\frac{1}{2}   \sigma(W,\phi)^2}, \quad \P\text{-a.s.},
\end{align*}
thus giving \eqref{cea6.7} as required.
\end{proof}

 The following sections will give the proofs of the three Lemmas \ref{al2.1}-\ref{al4.2}.

\section{Moment formulas}\label{as3}

To prove the remaining lemmas in Section \ref{as3}, we need the second-moment formulas for $Y_t$. Let $\mu \in \cM(\R^d)$. Recall from \eqref{ed1.1} the conditional Laplace transform that for any $\phi \in C_c^+(\R^d)$ and $\theta \geq 0$, we have
\begin{align} \label{e2.2}
 \E_\mu^W (e^{-\theta Y_t(\phi)  })= e^{-\la \mu, U^{\theta\phi}(t,\cdot)\ra},
\end{align}
where $U^{\theta\phi}\geq 0$ satisfies
\begin{align}\label{e9.1}
U^{\theta\phi}(t,x)=  \theta  Q_t \phi(x)  &-\frac{1}{2}\int_0^t  ds \int p_{t-s}(x,z) (U^{\theta\phi}(s,z))^2  dz\nn \\
&+\int_0^t  \int p_{t-s}(x,z)  U^{\theta\phi}(s,z) W(ds,z) dz.
\end{align}
For each $n\geq 1$, define
\begin{align}\label{e2.1}
V_n^\phi(t, x)=(-1)^{n-1}\frac{\partial^n}{\partial \theta^n} U^{\theta \phi}(t,x)|_{\theta=0}.
\end{align}
  By differentiating \eqref{e9.1} with respect to $\theta$ and letting $\theta=0$, one can check that $V_1^\phi(t)$ satisfies 
\begin{align}\label{e10.5}
V_1^\phi(t,x)=  &Q_{t} \phi(x) +\int_0^t  \int p_{t-s}(x,z)  V_1^\phi(s, z) W(ds,z) dz.
\end{align}
Similar to the derivation of Lemma 2.1, we may iteratively differentiate \eqref{e9.1} with respect to $\theta$ to get that for any $n\geq 2$,  
\begin{align}\label{ae6.11}
V_n^\phi(t,x)=  &\sum_{k=1}^{n-1} \binom{n-1}{k}   \int_0^t ds \int p_{t-s} (x,z)V_{n-k}^\phi(s,z)V_k^\phi(s,z) dz\nn\\
&+\int_0^t  \int p_{t-s}(x,z)  V_n^\phi(s, z) W(ds,z) dz.
\end{align}

Although we only need the second-moment formulas for the proof in the current work, we state the following lemmas of all moments for completeness. 

\begin{lemma}\label{l2.1}
For any $t\geq 0$ and $\phi\in C_c^+(\R)$, we have
\begin{align}\label{e10.22}
 \E_\mu^W [Y_t(\phi)^n]= L^{(n)}_t, \quad \forall n\geq 1,
\end{align}
where
\begin{align}\label{edd10.11}
L^{(0)}_t=&1, \quad L^{(1)}_t=   \la \mu, V_1^\phi(t)\ra; \\
L^{(n)}_t=&\sum_{k=0}^{n-1} \binom{n-1}{k}   \la \mu, V_{n-k}^\phi(t) \ra \cdot L_t^{(k)} , \quad \forall n\geq 2.\nn
\end{align}
\end{lemma}

\begin{proof}
 Taking the $n$-th derivative with respect to $\theta$ on both sides of \eqref{e2.2}, we obtain that
\begin{align}\label{ea9.2}
 \E_\mu^W \Big[(-Y_t(\phi))^n e^{-\theta Y_t(\phi)}\Big]=  \frac{\partial^n}{\partial \theta^n} e^{-\la \mu, U^{\theta\phi}(t)\ra}.
\end{align}
Define
\begin{align}\label{e9.2}
L^{(n)}_t:=(-1)^n \frac{\partial^n}{\partial \theta^n} e^{-\la \mu, U^{\theta \phi}(t)\ra}|_{\theta=0}.
\end{align}
Let $\theta=0$ in \eqref{ea9.2} to see that
 \begin{align*} 
 \E_\mu^W [Y_t(\phi)^n]=  L^{(n)}_t .
\end{align*}
To obtain \eqref{edd10.11}, we note that
 \begin{align*} 
 \frac{\partial}{\partial \theta} e^{-\la \mu, U^{\theta\phi}(t)\ra}=-\Big\la \mu,  \frac{\partial}{\partial \theta} U^{\theta\phi}(t)\Big\ra e^{-\la \mu, U^{\theta\phi}(t)\ra} 
\end{align*}
Hence, for any $n\geq 1$, by Leibniz rule (i.e., high order product rule), we get
 \begin{align*} 
 \frac{\partial^n}{\partial \theta^n} e^{-\la \mu, U^{\theta\phi}(t)\ra}&=-\sum_{k=0}^{n-1} \binom{n-1}{k}  \frac{\partial^{n-1-k}}{\partial \theta^{n-1-k}}  \Big\la \mu,  \frac{\partial}{\partial \theta} U^{\theta\phi}(t)\Big\ra  \frac{\partial^{k}}{\partial \theta^{k}}  e^{-\la \mu, U^{\theta\phi}(t)\ra} \\
 &=-\sum_{k=0}^{n-1} \binom{n-1}{k}   \Big \la \mu,  \frac{\partial^{n-k}}{\partial \theta^{n-k}} U^{\theta\phi}(t)\Big\ra  \frac{\partial^{k}}{\partial \theta^{k}}  e^{-\la \mu, U^{\theta\phi}(t)\ra}.
\end{align*}
Given \eqref{e2.1} and \eqref{e9.2}, we may let $\theta=0$ in the above to complete the proof.
\end{proof}

In particular, when $n=2$, the above lemma implies that
\begin{align}\label{e10.11}
 \E_\mu^W [Y_t(\phi)^2]= \la \mu, V_2^\phi(t)\ra+\la \mu, V_1^\phi(t)\ra^2. 
\end{align}
Hence, we need the following moment formulas for $V_1^\phi(t,x)$ and $V_2^\phi(t,x)$.

\begin{lemma}\label{l2.2}
For any $t\geq 0$ and $\phi, \psi \in C_c^+(\R^d)$, we have for all $x,y\in \R^d$,
\begin{align}\label{ea10.11}
 \E [V_1^\phi(t,x)]=  Q_t\phi(x), \quad  \E [V_1^\phi(t,x) V_1^\psi(t,y)]=  V^{\phi, \psi}_t(x,y),
 \end{align}
and
\begin{align}\label{ea10.12}
 \E [V_2^\phi(t,x)]=\int_0^t ds  \int p_{t-s}(x,z) V^{\phi, \phi}_s(z,z) dz,
 \end{align}
 where
\begin{align} \label{e10.23}
V^{\phi, \psi}_t(x,y):=\Pi_{(x,y)}\Big\{\int_0^t  \Big[\phi(B_s) Q_{t-s}\psi(\tilde{B}_s)+\psi(\tilde{B}_s) Q_{t-s}\phi({B}_s)\Big] e^{\int_0^s g(B_u, \tilde{B}_u) du} ds\Big\}.
\end{align}
and $B_t, \tilde{B}_t$ are independent Brownian motions starting respectively  from $x,y\in \R^d$ under $\Pi_{(x,y)}$.  
\end{lemma}

 \begin{proof}
 Let $t\geq 0$, $x,y\in \R^d$ and $\phi, \psi\in C_c^+(\R^d)$.
 It is immediate from \eqref{e10.5} that  $$\E [V_1^\phi(t,x)]=  Q_t\phi(x).$$ Next, use \eqref{e10.5} again to see that 
\begin{align*}  
\E( V_1^\phi(t,x) &  V_1^\psi(t,y))= Q_t \phi(x) \cdot Q_t \psi(y)\nn\\
&+\int_0^t ds\int_{\R^{2d}} p_{t-s}(x,z)p_{t-s}(y,w) \E( V_1^\phi(s,z)  V_1^\psi(s,w)) g(z,w) dz dw.
\end{align*}
Write $F_t(x,y)=\E( V_1^\phi(t,x)    V_1^\psi(t,y))$. One may easily derive from the above that
\begin{align} \label{e1.5}
 \frac{\partial F_t(x,y)}{\partial t}=\frac{\Delta_{(x,y)}}{2} F_t(x,y)+g(x,y) F_t(x,y)+  \phi(x)  Q_t \psi(y)+  \psi(y)  Q_t \phi(x).
\end{align}
By Feyman-Kac's formula applied to \eqref{e1.5}, along with $F_0(x,y)=0$ we get
\begin{align*} 
\E( V_1^\phi(t,x) &  V_1^\phi(t,y))=F_t(x,y)=V^{\phi, \psi}_t(x,y),
\end{align*}
where $V^{\phi, \psi}_t(x,y)$ is as in \eqref{e10.23}, and hence \eqref{ea10.11} follows.
 
 Next, by using \eqref{e6.11} with $n=2$, we get 
 \begin{align}\label{e6.11}
V_2^\phi(t,x)=  &  \int_0^t ds \int p_{t-s} (x,z)V_{1}^\phi(s,z)^2 dz\nn\\
&+\int_0^t  \int p_{t-s}(x,z)  V_2^\phi(s, z) W(ds,z) dz.
\end{align}
By taking expectations on both sides, we obtain \eqref{ea10.12} from \eqref{ea10.11}.
  \end{proof}

The following is immediate by combining the above two lemmas.
 
\begin{corollary}\label{c2.1}
For every $t\geq 0$ and $\phi \in C_c^+(\R^d)$, we have
\begin{align} \label{e3}
 \E_\mu [Y_t(\phi)]= \la \mu, Q_t\phi\ra,
\end{align}
and
\begin{align} \label{e10.7}
 \E_\mu [Y_t(\phi)^2]= \int &\int V^{\phi, \phi}_t(x,y)\mu(dx) \mu(dy)\nn\\
&+\int \mu(dx) \int_0^t ds \int p_{t-s}(x,y) V^{\phi, \phi}_s(y,y) dy,
\end{align}
 
\end{corollary}

 \section{Preliminary bounds of the second moments}\label{as4}

Using the moment formulas from the last section, we give the proof of Lemma  \ref{al4.1}  in this section. {\bf 
Throughout the rest of the paper, we fix $\phi \in C_c^+(\R^d)$ and let $K>2$ be large such that the function $\phi$ is supported on $\{x\in \R^d: |x|\leq K\}$.\\}

Recall from \eqref{ea10.11} that for any $t\geq 0$ and $x\in \R^d$,
\begin{align*}  
\E[V_1^\phi(t, x)^2]=V^{\phi, \phi}_t(x,x).
\end{align*}
Combining the monotonicity of $V_1^\phi(t, x)$ in $t$ (see \eqref{ea6.6}), we may reduce to proof of Lemma \ref{al4.1} to showing that when $d\geq 5$ and $\alpha>2$,
\begin{align}  \label{ae8.1}
\sup_{t\geq 0} \int V^{\phi, \phi}_t(x,x) dx <\infty.
\end{align}
The following bounds on $V^{\phi, \phi}_t(x,y)$ will be the key to the proof. \\

Let $p\in  (1,\frac{10}{9})$ and $q>1$ satisfy  $\frac{1}{p}+\frac{1}{q}=1$.  We will choose $p=p(d,\alpha)$ close to $1$ below.  So for $q=q(d,\alpha)>1$, we have \eqref{e1} from Lemma \ref{l1.1} holds by letting $g$ satisfy \eqref{e0.0} with $\eps$ small.\\

For each $t\geq 0$ and $x\in \R^d$, we define 
\begin{align}\label{ae9.41}
 \widetilde{Q}(t,x):=\int_0^t  \Big[\int_{|z|\leq K} p_s(x,z) dz\Big]^{1/p} ds,
\end{align}
where the dependence of $ \widetilde{Q}(t,x)$ on $p\in  (1,\frac{10}{9})$ and $K>0$ is suppressed for notation ease.\\

\no ${\bf Convention\ on\ Constants.}$ Constants whose value is unimportant and may change from line to line are denoted $C$. All these constants may depend on the dimension $d$, the covariance function $g$ and $\alpha$, the test function $\phi$, and the constant $p=p(d,\alpha)$. All these parameters, $d,\alpha, \phi, p$, will be fixed before picking $C$. Our constants $C$ will never depend on $t\geq 0$ or  $x\in \R^d$.\\

\begin{lemma} \label{al4.6}
Let  $d\geq 3$ and $\alpha>2$. For any $t\geq 0$ and $x,y \in \R^d$, we have
\[
V^{\phi, \phi}_t(x,y) \leq  C\cdot \widetilde{Q}(t,x)\widetilde{Q}(t,y).
\]
\end{lemma}
\begin{proof}
Fix any $t\geq 0$ and $x,y \in \R^d$.
Recall $V^{\phi, \phi}_t(x,y)$ from \eqref{e10.23}. By symmetry, it suffices to prove that
\begin{align*} 
I:=& \int_0^t    \Pi_{(x,y)}\Big\{ \phi(B_s) Q_{t-s}\phi(\tilde{B}_s) e^{\int_0^s g(B_u, \tilde{B}_u) du}  \Big\} ds\leq  C\cdot \widetilde{Q}(t,x)\widetilde{Q}(t,y).
\end{align*}
To do this, we note that (recall $Q_t\phi$ from \eqref{ae10.99})
\begin{align*} 
I=&  \int_0^t  ds \int_0^{t-s} dr   \Pi_{(x,y)}\Big\{ \phi(B_s) P_r\phi(\tilde{B}_s) e^{\int_0^s g(B_u, \tilde{B}_u) du}  \Big\} ds\\
\leq&  \int_0^t  ds \int_0^{t-s} dr   \Big[ \Pi_{(x,y)}\Big( \phi(B_s)^p [P_r\phi(\tilde{B}_s)]^p\Big)\Big]^{1/p}  \Big[\Pi_{(x,y)}\Big(e^{q  \int_0^s g(B_u, \tilde{B}_u) du} \Big)\Big]^{1/q},
\end{align*}
where the last inequality uses H\"older's inequality. Now apply Lemma \ref{l1.1} to bound the above by
\begin{align} \label{eb5.1}
I \leq& 2^{1/q} \int_0^t  ds \int_0^{t-s}  dr   \Big[ \Pi_{(x,y)}\Big( \phi(B_s)^p [P_r\phi(\tilde{B}_s)]^p\Big)\Big]^{1/p}  \nn\\
= & 2^{1/q} \int_0^t  ds \int_0^{t-s}  dr \  \Big[\int p_s(x,z) \phi(z)^p dz\Big]^{1/p}  \Big[\int p_s(y,w) [P_r\phi(w)]^p dw\Big]^{1/p}.
\end{align}
Use that $\phi$ is supported on $\{|x|\leq K\}$ to get 
\begin{align*} 
 \Big[\int p_s(x,z) \phi(z)^p dz\Big]^{1/p}  \leq \|\phi\|_\infty \Big[\int_{|z|\leq K} p_s(x,z) dz\Big]^{1/p}.
\end{align*}
Next,  apply Jensen's inequality to bound $[P_r\phi(w)]^p$ by $\int p_r(w,u) \phi(u)^p du$ to see that
\begin{align} \label{ae7.51}
\Big[\int p_s(y,w) [P_r\phi(w)]^p dw\Big]^{1/p}& \leq \Big[\int p_{s+r}(y,u) \phi(u)^p du\Big]^{1/p} \nn\\
&\leq  \|\phi\|_\infty \Big[\int_{|z|\leq K} p_{s+r}(y,u) du\Big]^{1/p}.
\end{align}

Returning to \eqref{eb5.1}, we conclude that 
\begin{align} \label{eb5.2}
I \leq  & C \int_0^t  ds \int_0^{t-s}  dr \  \Big[\int_{|z|\leq K} p_s(x,z) dz\Big]^{1/p} \Big[\int_{|z|\leq K} p_{s+r}(y,u) du\Big]^{1/p}\nn\\
  \leq   & C\Big(\int_0^t  \Big[\int_{|z|\leq K} p_s(x,z) dz\Big]^{1/p} ds  \Big) \Big(\int_0^t  \Big[\int_{|w|\leq K} p_r(y,w) dw\Big]^{1/p} dr  \Big),
\end{align}
as required.
\end{proof}

Define
\begin{align} \label{ae9.11}
I_t(x):=1\wedge\Big( \int_0^t s^{\frac{d}{2}-\frac{d}{2p}}  p_{8s}(x)   ds\Big), \quad t\geq 0, x\in \R^d.
\end{align}

\begin{lemma}\label{la4.1}
For all $t\geq 0$ and $x\in \R^d$, we have
\[
\widetilde{Q}(t,x)  \leq C I_t(x).
\]
\end{lemma}

  \begin{proof}
 Recall from \eqref{ae9.41} that 
 \begin{align*}
 \widetilde{Q}(t,x)=\int_0^t  \Big[\int_{|y|\leq K} p_s(x,y) dy\Big]^{1/p} ds.
\end{align*}
If $|x|\leq 2K$, then we have
\begin{align*} 
\widetilde{Q}(t,x) & \leq \int_0^1 1 ds +\int_1^{t \vee 1}  \Big[\int_{|y|\leq K} \frac{1}{s^{d/2}} dy\Big]^{1/p} ds\\
&\leq 1+C \int_1^{t \vee 1}  {s^{-\frac{d}{2p}}} ds \leq C,
\end{align*}
where the last inequality follows by $\frac{d}{2p}>1$ (recall $p<\frac{5}{4}\leq \frac{d}{2}$ for $d\geq 3$). \\

Turning to $|x|>2K$, we use $|y-x|\geq |x|-K\geq |x|/2$ to get 
 \begin{align*} 
\widetilde{Q}(t,x) & \leq  \int_0^t  \Big[\int_{|y|\leq K} \frac{1}{s^{d/2}} e^{-\frac{|x|^2}{8s}} dy\Big]^{1/p} ds\\
&\leq C \int_0^t s^{-\frac{d}{2p}}  e^{-\frac{|x|^2}{8ps}}   ds\leq C   \int_0^t s^{\frac{d}{2}-\frac{d}{2p}}  p_{8s}(x)   ds.
\end{align*}
The conclusion follows by adjusting the constant $C>0$.
\end{proof}

The following is the final piece needed for the proof of Lemma \ref{al4.1}.

\begin{corollary}\label{ac4.6}
Let  $d\geq 3$ and $\alpha>2$. For any $t\geq 0$ and $x,y \in \R^d$, we have
\[
V^{\phi, \phi}_t(x,y) \leq  CI_t(x) I_t(y) \leq C(|x|^{2-\frac{d}{p}} \wedge 1) (|y|^{2-\frac{d}{p}} \wedge 1).
\]
\end{corollary}
\begin{proof}
The first inequality is immediate from Lemma \ref{al4.6} and Lemma \ref{la4.1}. To check the second inequality, we notice that if $|x|>1$, then for any $t\geq 0$,
\begin{align} \label{ae8.25}
 &  \int_0^t s^{\frac{d}{2}-\frac{d}{2p}}  p_{8s}(x)   ds \leq C\int_0^\infty s^{-\frac{d}{2p}}  e^{-\frac{|x|^2}{16s}}   ds\nn\\
&= C|x|^{2-\frac{d}{p}} \int_0^\infty r^{\frac{d}{2p}-2}  e^{-r}   dr \leq C |x|^{2-\frac{d}{p}},
\end{align}
where the last inequality follows by $\frac{d}{2p}-2>-1$. So the proof is complete
\end{proof}

We can prove Lemma \ref{al4.1} using the above result.

\begin{proof}[Proof of Lemma \ref{al4.1}]
Let $d\geq 5$ and $\alpha>2$.  Apply Corollary \ref{ac4.6} to get
\begin{align} \label{ae8.59}
\int V^{\phi, \phi}_t(x,x) dx \leq      C +C\int_{|x|>1} |x|^{4-\frac{2d}{p}} dx\leq  C+ C\int_{1}^\infty r^{4-\frac{2d}{p}} r^{d-1} dr <\infty,
\end{align}
where the last inequality is due to 
\begin{align*}
p<\frac{10}{9}\leq \frac{2d}{4+d} \quad \Rightarrow  \quad 4-\frac{2d}{p}+d-1<-1.
\end{align*}
 Hence, \eqref{ae8.1}  follows. The proof is complete, as noted above.
\end{proof}

  \section{Estimates of the second moments}\label{as5}

We will prove Lemma \ref{al2.1} in this section, thus finishing the proof of Theorem \ref{t1}. Fix $d\geq 3$ and $\alpha>2$.  
By using
\begin{align*}
 \Big(Y_T(\phi)-T \la \lambda, \phi \ra\Big)^2\leq 2 \Big(Y_T(\phi)-\la \lambda, V_1^\phi(T) \ra\Big)^2 +2\Big(\la \lambda, V_1^\phi(T) \ra-T \la \lambda, \phi \ra\Big)^2,
\end{align*}
we may reduce the proof of Lemma \ref{al2.1} to proving that there is some $\delta\in (0,1)$ such that
\begin{align}\label{ae2.11}
\E \Big[\Big(Y_T(\phi)-\la \lambda, V_1^\phi(T) \ra\Big)^2\Big]\leq CT^{2-\delta}, \quad \forall T>0,
\end{align}
and 
\begin{align}\label{ae2.10}
\E \Big[\Big(\la \lambda, V_1^\phi(T) \ra-T \la \lambda, \phi \ra\Big)^2\Big]\leq CT^{2-\delta}, \quad \forall T>0,
\end{align}

 We will prove the above in the following two steps.\\

{\it Step 1.}  
By using \eqref{ea6.2} and \eqref{e10.11}, one can check that 
\begin{align} 
 \E^W \Big[  \Big(Y_T(\phi)-\la \lambda, V_1^\phi(T) \ra\Big)^2\Big]=  \la \lambda, V_2^\phi(T)\ra,
\end{align}
and hence  Lemma \ref{l2.2} implies that
\begin{align}  \label{ea4.3}
 &\E  \Big[  \Big(Y_T(\phi)-\la \lambda, V_1^\phi(T) \ra\Big)^2\Big]=  \E \Big[ \la \lambda, V_2^\phi(T)\ra \Big]\nn\\
 &= \int dx \int_0^T dt  \int p_{T-t}(x,z) V^{\phi, \phi}_t(z,z) dz = \int_0^T dt \int    V^{\phi, \phi}_t (z, z) dz,
\end{align}
where in the last equality, we have used Fubini's theorem and that $\int p_{T-t}(x,z) dx=1$.\\

When $d\geq 5$, we get from Lemma \ref{al4.1} that \eqref{ea4.3} becomes
\begin{align} \label{ae2.12}
 \E  \Big[  \Big(Y_T(\phi)-\la \mu, V_1^\phi(T) \ra\Big)^2\Big]   \leq CT,
\end{align}
as required.\\

When $d=3$ or $4$, we apply Lemma \ref{al4.6} and Lemma \ref{la4.1} to see that
\begin{align} \label{ea4.2}
\int V^{\phi, \phi}_t(x,x) dx \leq&     C +C\int_{|x|>1} \Big[\int_0^t s^{\frac{d}{2}-\frac{d}{2p}}  p_{8s}(x)   ds\Big]^2 dx.
\end{align}
The integral on the right-hand side of \eqref{ea4.2} is equal to
\begin{align*} 
&\int_0^t s^{\frac{d}{2}-\frac{d}{2p}} ds \int_0^t  r^{\frac{d}{2}-\frac{d}{2p}} dr  \int_{|x|>1}  p_{8s}(x)  p_{8r}(x)    dx\\
&\leq C \int_0^t s^{\frac{d}{2}-\frac{d}{2p}} ds \int_0^t  r^{\frac{d}{2}-\frac{d}{2p}} dr  \cdot \frac{1}{(s+r)^{d/2}}.
\end{align*}
By letting $u=s+r$ and $s=s$, the right-hand side above is equal to
\begin{align} \label{ae2.1}
&C \int_0^{2t}  du \int_0^u    s^{\frac{d}{2}-\frac{d}{2p}}   (u-s)^{\frac{d}{2}-\frac{d}{2p}}  \cdot \frac{1}{u^{d/2}}  ds\nn\\
&\leq C\int_0^{2t}   u^{-\frac{d}{2p}} du \int_0^u   s^{\frac{d}{2}-\frac{d}{2p}}    ds \leq C\int_0^{2t}   u^{1+\frac{d}{2}-\frac{d}{p}} du.
\end{align}
Since $p>1\geq \frac{2d}{4+d}$ for $d=3$ or $4$, we get 
\begin{align*}
1+\frac{d}{2}-\frac{d}{p}>-1,
\end{align*}
and so \eqref{ae2.1} is at most $Ct^{2+\frac{d}{2}-\frac{d}{p}}$. It follows that \eqref{ea4.2} becomes
\begin{align*}
\int V^{\phi, \phi}_t(x,x) dx \leq&     C +Ct^{2+\frac{d}{2}-\frac{d}{p}}\leq C  (t \vee 1)^{2+\frac{d}{2}-\frac{d}{p}}.
\end{align*}

%{\Large Today }\\

Recalling \eqref{ea4.3}, we obtain that for any $T>1$,
\begin{align} \label{ae2.13}
 \E_\lambda \Big[  \Big(Y_T(\phi)-\la \lambda, V_1^\phi(T) \ra\Big)^2\Big] \leq \int_0^T C_K (t \vee 1)^{1+\frac{d}{2}-\frac{d}{p}} dt  \leq CT^{3+\frac{d}{2}-\frac{d}{p}}\leq CT^{3/2},
\end{align}
where the last inequality follows by (recall $p\in (1, \frac{10}{9})$ and $d=3$ or $4$)
\[
\frac{d}{2}-\frac{d}{p} \in (-2, -\frac{3}{2}).
\]
 
The proof of \eqref{ae2.11} is now complete in view of \eqref{ae2.12} and \eqref{ae2.13}.\\

{\it Step 2.} 
Recall $V_1^\phi(t,x)$ from \eqref{ae1.33} to see that
\begin{align}  \label{ea5.1}
\la \lambda, V_1^\phi(T) \ra=T \la \lambda, \phi \ra+\int_0^T \int V_1^\phi(t,y) W(dt, y)dy.
\end{align}
We used Fubini's theorem and stochastic Fubini's theorem for equality above.
It follows that
\begin{align}  \label{ae5.1}
  &\E \Big[\Big(\la \lambda, V_1^\phi(T) \ra-T \la \lambda, \phi \ra\Big)^2\Big]  \leq \E\Big[\int_0^T dt \int  \int V_1^\phi(t,x)V_1^\phi(t,y) g(x,y) dx dy \Big]\nn\\
  &=\int_0^T dt \int  \int V^{\phi, \phi}_t(x,y) g(x,y) dx dy \leq C \int_0^T dt \int  \int I_t(x)I_t(y) g(x,y) dx dy, 
\end{align}
where the equality follows by \eqref{ea10.11} and the last inequality uses Corollary \ref{ac4.6}.\\

 It remains to bound the last integral.

 \begin{lemma}\label{al5.1}
 Let $d\geq 3$ and $\alpha>2$. \\

\no (a) For any $2<\gamma< d\wedge \alpha \wedge 4$  and any $p \in (1, \frac{d}{d-(\gamma-2)})$, there is some constant $C>0$  depending on $d$, $\alpha$, $\gamma$, $p$ and $\phi$ such that
  \begin{align*}   
  \int_{\R^d}  \int_{\R^d}  I_t(x) I_t(y) g(x,y) dx dy \leq Ct^{2+d-\frac{d}{p}-\frac{\gamma}{2}}, \quad \forall t\geq 1.
 \end{align*}  
 \no  (b) If $d\wedge \alpha>4$, then 
 \begin{align*}   
  \int_{\R^d}  \int_{\R^d}  I_t(x) I_t(y) g(x,y) dx dy \leq C , \quad \forall t\geq 1.
 \end{align*}  
 
 \end{lemma}
 
Assuming the above, we first finish the proof of Lemma \ref{al2.1}.

\begin{proof}[Proof of Lemma \ref{al2.1}]

 When $d\geq 3$ and $\alpha>2$, we let $\gamma=2+\delta$ for some $\delta\in (0, \frac{\alpha-2}{2}\wedge \frac{1}{2})$ such that $\gamma<d\wedge \alpha \wedge 4$. Pick $p \in (1, \frac{10}{9} \wedge \frac{d}{d-\frac{\delta}{2}})$ so that
 \[
d-\frac{d}{p}-\frac{\delta}{2}<0.
\]
Apply Lemma \ref{al5.1} (a) with $\gamma=2+\delta$ to see that \eqref{ae5.1} with $T>1$ becomes
 \begin{align*}  
 &  \E \Big[\Big(\la \lambda, V_1^\phi(T) \ra-T \la \lambda, \phi \ra\Big)^2\Big]  \leq   C+\int_1^T C t^{2+d-\frac{d}{p}-\frac{2+\delta}{2}} dt \leq CT^{2+d-\frac{d}{p}-\frac{\delta}{2}}.
\end{align*}
Therefore \eqref{ae2.10} follows and the proof of Lemma \ref{al2.1} is complete given \eqref{ae2.11} and \eqref{ae2.10}.
\end{proof}

It remains to prove Lemma \ref{al5.1}.\\

To bound the integral in \eqref{ae5.1}, we note that
\begin{align*} 
\iint    I_t(x) I_t(y) g(x,y) dx dy&\leq  \int_{|x|\leq 4}  \int_{|y|\leq 4} I_t(x) I_t(y) g(x,y) dx dy\\
&+ \int_{|x|\leq 2}  \int_{|y|> 4} I_t(x) I_t(y) g(x,y) dx dy\\
&+ \int_{|x|> 4}  \int_{|y|\leq 2} I_t(x) I_t(y)  g(x,y) dx dy\\
&+ \int_{|x|> 2}  \int_{|y|> 2} I_t(x) I_t(y)  g(x,y) dx dy=: I_1+I_2+I_3+I_4.
\end{align*}
We will use Lemma \ref{al4.6} and Lemma \ref{la4.1} to bound the above integrals.\\

First, to bound $I_1$, we simply use  $I_t(x) I_t(y)$ by $1$ and  $g(x,y)\leq C$ to  see that 
\begin{align}  \label{ea4.11}
 &I_1\leq C.
\end{align}

Turning to $I_2$, we apply Corollary \ref{ac4.6}  to see that
\begin{align*} 
I_t(x) I_t(y) \leq     C |y|^{2-\frac{d}{p}}.
\end{align*}
 Noticing that when $|x|\leq 2$ and $|y|> 4$, we have $|x-y|\geq |y|-|x|\geq |y|/2$, thus giving $g(x,y)\leq C|y|^{-\alpha}$. We conclude that
\begin{align}  \label{ea5.77}
 I_2&\leq    C      \int_{|y|> 4} |y|^{-\alpha}|y|^{2-\frac{d}{p}} dy=C\int_{4}^\infty r^{-\alpha+2-\frac{d}{p}+d-1} dr.
\end{align}
Since $\alpha>2$, we may take $p>1$ such that  $p(d-\alpha+2)<d$. It follows that 
\begin{align*} 
-\alpha+2-\frac{d}{p}+d-1<-1,
\end{align*}
and hence \eqref{ea5.77} becomes
\begin{align}  \label{ea4.12}
 I_2&\leq    C.
\end{align}
 
By symmetry, we get that
\begin{align}  \label{ea4.13}
 I_3=I_2  \leq C.
 \end{align}
 
For the last integral $I_4$,  we have
 \begin{align*}  
 I_4 &\leq    C  \int_0^t s^{\frac{d}{2}-\frac{d}{2p}}   ds \int_0^t r^{\frac{d}{2}-\frac{d}{2p}} dr  \int_{\R^d}   \int_{\R^d}    p_{8s}(x)       p_{8r}(y) (|x-y|^{-\alpha} \wedge 1) dx dy. 
\end{align*}
By letting $z=x-y, x=x$ and applying Chapman–Kolmogorov's equation, we get
 \begin{align}  \label{ea4.6}
 I_4 &\leq  C     \int_0^t s^{\frac{d}{2}-\frac{d}{2p}}   ds \int_0^t r^{\frac{d}{2}-\frac{d}{2p}} dr  \int_{\R^d}     p_{8s+8r}(z)   (|z|^{-\alpha} \wedge 1) dz:=C   \cdot J. 
\end{align}
It suffices to bound $J$.\\

If $t\leq 1$, we get
\begin{align}  \label{ae7.61}
 & J \leq \int_0^1 s^{\frac{d}{2}-\frac{d}{2p}}   ds \int_0^1 r^{\frac{d}{2}-\frac{d}{2p}} dr\leq 1.
\end{align}
If $t>1$, we use \eqref{ae7.61} to see that
\begin{align}  \label{ea4.7}
J \leq 1 & +2 \int_0^1 s^{\frac{d}{2}-\frac{d}{2p}}   ds \int_1^t r^{\frac{d}{2}-\frac{d}{2p}} dr \int_{\R^d}     p_{8s+8r}(z)   (|z|^{-\alpha} \wedge 1) dz\nn\\
 &+ \int_1^t s^{\frac{d}{2}-\frac{d}{2p}}   ds \int_1^t r^{\frac{d}{2}-\frac{d}{2p}} dr \int_{\R^d}     p_{8s+8r}(z)   (|z|^{-\alpha} \wedge 1) dz:=1+2J_1+J_2.
\end{align}

\begin{lemma}\label{al5.4}
For any $t\geq 1$ and $\alpha>0$, we have 
\begin{align*}  
 \int_{\R^d} p_t(x)(|x|^{-\alpha} \wedge 1) dx \leq 
 \begin{cases}
 Ct^{-\alpha/2}, &\text{ if } \alpha<d;\\
  Ct^{-d/2} \log (2t), &\text{ if } \alpha=d;\\
   Ct^{-d/2}, &\text{ if } \alpha>d;\\
 \end{cases}
\end{align*}
\end{lemma}
\begin{proof}
Set
\begin{align*}  
I_{\alpha,d}(t):= \int_{\R^d} p_t(x)(|x|^{-\alpha} \wedge 1) dx.
\end{align*}
We have
\begin{align*}  
I_{\alpha,d}(t)& \leq C   \int_0^\infty (r^{-\alpha} \wedge 1) \frac{1}{ t^{d/2}} e^{-\frac{r^2}{2t}} r^{d-1} dr \\
&\leq C t^{-d/2}+C t^{-d/2} \int_1^\infty r^{d-1-\alpha}   e^{-\frac{r^2}{2t}}   dr.
\end{align*}
Denote the integral on the right-hand side above by
\begin{align}  \label{ae6.1}
J_{\alpha,d}(t):= \int_1^\infty r^{d-1-\alpha}   e^{-\frac{r^2}{2t}}   dr.
\end{align} 
 By using a change of variable, we get
\begin{align*}  
 J_{\alpha,d}(t)&=\frac{1}{2}  \int_{t^{-1}}^\infty (tx)^{\frac{d-1-\alpha}{2}}   e^{-\frac{x}{2}}  t^{\frac{1}{2}} x^{-\frac{1}{2}} dx\\
 &   =2^{-1} t^{\frac{d-\alpha}{2}}\int_{t^{-1}}^\infty x^{\frac{d-2-\alpha}{2}}   e^{-\frac{x}{2}}   dx.
 \end{align*}
If $\alpha<d$, then $J_{\alpha,d}(t)$ is bounded by $C   t^{\frac{d-\alpha}{2}}$ and hence
\begin{align*}  
I_{\alpha,d}(t)\leq C t^{-d/2}+C  t^{-d/2} \cdot C   t^{\frac{d-\alpha}{2}} \leq   C  t^{-\alpha/2}.
\end{align*}
If $\alpha=d$, then
\begin{align*}
	J_{\alpha,d}(t)=&\frac{1}{2}\int_{t^{-1}}^\infty x^{-1}   e^{-\frac{x}{2}}   dx\\
	\leq & \frac{1}{2} \int_{t^{-1}}^1 x^{-1}     dx +\frac{1}{2}\int_1^\infty   e^{-\frac{x}{2}}   dx  \leq  \frac{\log t}{2}+1.
\end{align*}
It follows that
\begin{align*}  
I_{\alpha,d}(t)\leq  C  t^{-d/2}+C  t^{-d/2} \cdot \Big(\frac{\log t}{2}+1\Big) \leq  C t^{-d/2}\log (2t).
\end{align*}
If $\alpha>d$, then we recall \eqref{ae6.1} to see that
  \begin{align*}  
 J_{\alpha,d}(t)=\int_1^\infty r^{d-1-\alpha}     dr \leq C,
\end{align*}
and hence
\begin{align*}  
I_{\alpha,d}(t)\leq  C t^{-d/2}+C  t^{-d/2} \cdot C  \leq   C t^{-d/2}.
\end{align*}
The proof is complete by adjusting the constant.
\end{proof}

We are ready to finish the proof of Lemma \ref{al5.1}.

 \begin{proof}[Proof of Lemma \ref{al5.1}]
For any $\alpha>2$ and $d\geq 3$, we let  $2<\gamma< d\wedge \alpha \wedge 4$.  For all $t\geq 1$, one can check by Lemma \ref{al5.4} that in all three cases, we have
\begin{align}   \label{ea4.21}
 \int_{\R^d} p_t(x)(|x|^{-\alpha} \wedge 1) dx \leq C  t^{-\frac{\gamma}{2}}.
\end{align}

Apply the above to get that (recall $J_1$ from \eqref{ea4.7})
\begin{align}  \label{ea4.8}
 J_1 &\leq  C \int_0^1 s^{\frac{d}{2}-\frac{d}{2p}}   ds \int_1^t r^{\frac{d}{2}-\frac{d}{2p}} (s+r)^{-\frac{\gamma}{2}} dr \nn\\
 &\leq C    \int_1^t r^{\frac{d}{2}-\frac{d}{2p}-\frac{\gamma}{2}}   dr \leq C,
\end{align}
where the second inequality follows by $s^{\frac{d}{2}-\frac{d}{2p}}  (s+r)^{-\frac{\gamma}{2}} \leq (2r)^{-\frac{\gamma}{2}}$ for $0\leq s\leq 1 \leq r$, and the last inequality follows by letting $p \in (1, \frac{d}{d-(\gamma-2)})$ close to $1$ such that
\[
\frac{d}{2}-\frac{d}{2p}-\frac{\gamma}{2}<-1.
\]

Turning to $J_2$, we use \eqref{ea4.21} to get
\begin{align*}  
J_2 &\leq  C  \int_1^t s^{\frac{d}{2}-\frac{d}{2p}}   ds \int_1^t r^{\frac{d}{2}-\frac{d}{2p}} (s+r)^{-\frac{\gamma}{2}} dr \nn\\
 &=C  \int_2^{2t} u^{-\frac{\gamma}{2}} du   \int_1^u s^{\frac{d}{2}-\frac{d}{2p}}   (u-s)^{\frac{d}{2}-\frac{d}{2p}}  ds \nn\\
 &\leq C \int_2^{2t} u^{-\frac{\gamma}{2}}     u^{1+d-\frac{d}{p}}  du,
 \end{align*}
 where the equality follows by letting $u=s+r$ and $s=s$. Again we pick $p \in (1, \frac{2d}{2d-(\gamma-2)})$ close to $1$ such that
\[
1+d-\frac{d}{p}-\frac{\gamma}{2}<0.
\]
Moreover, since $\gamma<4$, we get
\[
1+d-\frac{d}{p}-\frac{\gamma}{2}>-1,
\]
and hence
 \begin{align}  \label{ea4.9}
J_2 &\leq   C  t^{2+d-\frac{d}{p}-\frac{\gamma}{2}}.
 \end{align}
 
  Combine \eqref{ea4.7},  \eqref{ea4.8} and  \eqref{ea4.9} to see that if $t>1$, then 
\begin{align}  \label{ea4.10}
J\leq      C+C t^{2+d-\frac{d}{p}-\frac{\gamma}{2}}\leq Ct^{2+d-\frac{d}{p}-\frac{\gamma}{2}},
 \end{align}
 where the last inequality follows by $\gamma<4$.

 Together with \eqref{ae7.61} for $t\leq 1$, we conclude that (recall \eqref{ea4.6})
  \begin{align}   
 I_4 \leq C \cdot J\leq C (t\vee 1)^{2+d-\frac{d}{p}-\frac{\gamma}{2}}.
\end{align} 
 
  Finally, we collect the four pieces \eqref{ea4.11}, \eqref{ea4.12}, \eqref{ea4.13} and the above to see that
  \begin{align*}   
  \int_{\R^d}  \int_{\R^d}  I_t(x) I_t(y) g(x,y) dx dy \leq C(t\vee 1)^{2+d-\frac{d}{p}-\frac{\gamma}{2}}.
 \end{align*} 
 The proof of (a) is now complete.\\

 The proof of (b) follows similarly. Pick $\gamma' \in (4, d\wedge \alpha)$ such that for all $t\geq 1$, one can check by Lemma \ref{al5.4} that in all three cases, we have
\begin{align}   \label{ea4.216}
 \int_{\R^d} p_t(x)(|x|^{-\alpha} \wedge 1) dx \leq C  t^{-\frac{\gamma'}{2}}.
\end{align}
Repeat the above calculations to bound $J_1$ by $C$ and $J_2$ by $C$ and so $I_4\leq C$. It follows that
\begin{align*}   
  \int_{\R^d}  \int_{\R^d}  I_t(x) I_t(y) g(x,y) dx dy \leq C,
 \end{align*} 
 as required.
 \end{proof}

\section{Convergence of the conditional Laplace transform}\label{as6}

We will prove Lemma \ref{al4.2} in this section, that is,
\begin{align}\label{ea6.21}
  \lim_{T\to \infty} \E[|I_2(T,\phi)|]=0, \quad \text{ and } \quad  \lim_{T\to \infty} \E[I_3^2(T,\phi)]=0.
\end{align}
where $I_2(T,\phi)$ and $I_3(T,\phi)$ are as in \eqref{ae7.32}.

\subsection{Convergence of $I_2(T,\phi)$}
Recall from \eqref{ae7.32} that
\begin{align*}
I_2(T, \phi)=\frac{1}{2T}   \int_0^T dt \int  \Big(V_1^\phi(t, x)^2-u_T^\phi(t,x)^2\Big) dx.
\end{align*}
where $u_T^\phi(t,x)$ is defined as in \eqref{ea6.1}. In view of \eqref{ae7.31}, we get 
\begin{align*} 
0\leq I_2(T, \phi)= &\frac{1}{2T}  \int_0^T ds \int  \Big(V_1^\phi(t, x)^2-[V_1^\phi(t, x)-v_T^\phi(t, x)]^2\Big) dx\\
\leq &\frac{1}{T} \int_0^T dt \int  V_1^\phi(t, x)v_T^\phi(t, x)  dx.
  \end{align*}
  It suffices to prove that 
\begin{align}\label{ae7.33}
  \lim_{T\to \infty}\frac{1}{T} \int_0^T dt  \int  \E[ V_1^\phi(t, x)v_T^\phi(t, x)]  dx=0.
\end{align}
  By using \eqref{ae7.34} and $u_T^\phi(t,x)\leq V_1^\phi(t,x)$, we get
 \begin{align}\label{ae7.35}
0\leq v_T^\phi(t,x)\leq \frac{1}{T^{1/2}} & \int_0^t ds \int p_{t-s}(x,y) V_1^\phi(s,y)^2 dy\nn\\
&+\int_0^t ds \int p_{t-s}(x,y) v_T^\phi(s,y) \dot{W}(s,y) dy.
\end{align}
Hence for any $x,y\in \R^d$ and $t>0$, we get
\begin{align*} 
0\leq &V_1^\phi(t, x)v_T^\phi(t, y)\leq \Big[Q_t \phi(x)+ \int_0^t ds \int p_{t-s}(x,z) V_1^\phi(s,z) \dot{W}(s,z) dz\Big]\\
&\times \Big[  \frac{1}{T^{1/2}}  \int_0^t ds \int p_{t-s}(y,w) V_1^\phi(s,w)^2 dw+\int_0^t ds \int p_{t-s}(y,w) v_T^\phi(s,w) \dot{W}(s,w) dw\Big].
\end{align*}
By taking expectations, we have
\begin{align}\label{ae7.36} 
 \E[V_1^\phi(t, x)& v_T^\phi(t, y)]  \leq    T^{-1/2}  Q_t \phi(x)  \int_0^t ds \int p_{t-s}(y,w) \E[V_1^\phi(s, w)^2] dw\nn \\
 &+ \int_0^t ds \int \int p_{t-s}(x,z)   p_{t-s}(y,w) \E[V_1^\phi(s, z)  v_T^\phi(s,w)] g(z,w) dz dw.
\end{align}

Define  
 \begin{align*} 
 H(t,y)= \int_0^t ds \int p_{t-s}(y,w) \E[V_1^\phi(s, w)^2] dw=\int_0^t ds \int p_{t-s}(y, w)  V_s^{\phi,\phi}(w,w)dw.
\end{align*}
Set $F(t, x, y)$ to be the solution of 
\begin{align}\label{ae7.37} 
 F(t, x, y)  =   & T^{-1/2}   Q_t \phi(x)  H(t,y) \nn\\
 &+ \int_0^t ds \int \int p_{t-s}(x,z)   p_{t-s}(y,w)  F(s, z, w) g(z,w) dz dw.
\end{align}

The following comparison lemma allows us to bound $\E[V_1^\phi(t, x) v_T^\phi(t, y)]$ by $ F(t, x, y) $.

  \begin{lemma}\label{l4.3}
Let $d\geq 1$. For any two continuous functions $F(t,x)$ and $G(t,x)$ defined on $[0, \infty) \times \R^d$, if there is some continuous function $\{\alpha(t,x): t\geq 0, x\in \R^d\}$ such that for all $t\geq 0$ and $x\in \R^d$,
\begin{align*}
 G(t,x)\leq \alpha(t,x)+ \int_0^t ds \int p_{t-s}(x,y) G(s,y)dy,
\end{align*}
 and
\begin{align*}
F(t,x)= \alpha(t,x)+ \int_0^t ds \int p_{t-s}(x,y) F(s,y)dy,
\end{align*}
 then 
   \begin{align}\label{e5.45}
 G(t,x)\leq F(t,x), \quad \forall  t\geq 0, x\in \R^d.
\end{align}
\end{lemma}
 \begin{proof}
The proof is deferred to Appendix \ref{a2}.
 \end{proof}

In view of \eqref{ae7.36} and \eqref{ae7.37}, we apply Lemma \ref{l4.3} to conclude that 
\[
 \E[V_1^\phi(t, x) v_T^\phi(t, y)] \leq F(t, x, y), \quad \forall t\geq 0, x,y\in \R^d.
\]
Therefore \eqref{ae7.33} is reduced to proving that
\begin{align}\label{ae7.38}
  \lim_{T\to \infty}\frac{1}{T} \int_0^T dt  \int  F(t, x, x)  dx=0.
\end{align}

One can easily check that
 \begin{align*} 
 \frac{\partial}{\partial t}H(t,y)= \frac{\Delta_y}{2} H(t,y)+V_t^{\phi,\phi}(y,y).
\end{align*}
Hence we get
 \begin{align} \label{ea9.91}
 \frac{\partial}{\partial t} F(t, x, y)&=   T^{-1/2}  P_t \phi(x)  H(t,y)+  T^{-1/2}  Q_t \phi(x) \Big[\frac{\Delta_y}{2} H(t,y)+V_t^{\phi,\phi}(y,y)\Big]\nn\\
 &+\frac{\Delta_{(x,y)}}{2} \Big[F(t, x, y)-T^{-1/2}  Q_t \phi(x) H(t,y)\Big]+F(t, x, y) g(x,y).
\end{align}
Notice that
 \begin{align*} 
 \frac{\Delta_{(x,y)}}{2} & \Big[T^{-1/2}Q_t \phi(x) H(t,y)\Big]= T^{-1/2}H(t,y) \frac{\Delta_{x}}{2} Q_t \phi(x)+ T^{-1/2} Q_t \phi(x) \frac{\Delta_{y}}{2} H(t,y)\\
 &=T^{-1/2} H(t,y) \Big[ P_t\phi(x)-\phi(x)\Big]+  T^{-1/2}Q_t \phi(x) \frac{\Delta_{y}}{2} H(t,y).
\end{align*}
Plug in the above  to see that \eqref{ea9.91} becomes
 \begin{align} \label{ea9.99}
 \frac{\partial}{\partial t} F(t, x, y)&= \frac{\Delta_{(x,y)}}{2} F(t, x, y)  +F(t, x, y) g(x,y)\nn\\
  &+  T^{-1/2}  Q_t \phi(x)  V_t^{\phi,\phi}(y,y) +T^{-1/2}   \phi(x)H(t,y).
\end{align}
By Feynman-Kac's formula, we conclude that
 \begin{align*} 
 F(t, x, y)&=T^{-1/2}\int_0^t   \Pi_{(x,y)} \Big[  Q_{t-s} \phi(B_s)  V_{t-s}^{\phi,\phi}(\widetilde{B}_s,\widetilde{B}_s) e^{\int_0^s g(B_u, \widetilde{B}_u) du}\Big] ds\\
 &+ T^{-1/2}\int_0^t   \Pi_{(x,y)} \Big[   \phi(B_s)H({t-s},\widetilde{B}_s) e^{\int_0^s g(B_u, \widetilde{B}_u) du} \Big] ds.
\end{align*}
Define
\begin{align}\label{ae8.28}
F_1(t,x)=&\int_0^t   \Pi_{(x,x)} \Big[  Q_{t-s} \phi(B_s)  V_{t-s}^{\phi,\phi}(\widetilde{B}_s,\widetilde{B}_s) e^{\int_0^s g(B_u, \widetilde{B}_u) du}\Big] ds,\nn\\
F_2(t,x)=&\int_0^t   \Pi_{(x,x)} \Big[   \phi(B_s)H({t-s},\widetilde{B}_s) e^{\int_0^s g(B_u, \widetilde{B}_u) du} \Big] ds.
\end{align}
Using the above, we may reduce the proof of \eqref{ae7.38} to proving
\begin{align}\label{ae7.39}
  &\lim_{T\to \infty}T^{-3/2} \int_0^T dt  \int  F_1(t,x)  dx=0,
  \end{align}
  and
\begin{align}\label{ae8.27}
 \lim_{T\to \infty}T^{-3/2} \int_0^T dt  \int  F_2(t,x)  dx=0.
\end{align}

We will prove the above in two steps.\\

\no {\bf Step 1.} For any $x\in \R^d$ and $t\geq 0$, we have
 \begin{align*} 
&\Pi_{(x,x)} \Big[  Q_{t-s} \phi(B_s)  V_{t-s}^{\phi,\phi}(\widetilde{B}_s,\widetilde{B}_s) e^{\int_0^s g(B_u, \widetilde{B}_u) du}\Big]\\
&=\int_0^{t-s} \Pi_{(x,x)} \Big[  P_r \phi(B_s)  V_{t-s}^{\phi,\phi}(\widetilde{B}_s,\widetilde{B}_s) e^{\int_0^s g(B_u, \widetilde{B}_u) du}\Big] dr\\
&\leq \int_0^{t-s} \Big(\Pi_{(x,x)} \Big[  (P_r \phi(B_s))^p  V_{t-s}^{\phi,\phi}(\widetilde{B}_s,\widetilde{B}_s)^p\Big]\Big)^{1/p} \Big(\Pi_{(x,x)} \Big[  e^{q\int_0^s g(B_u, \widetilde{B}_u) du}\Big]\Big)^{1/q} dr\\
&\leq C\int_0^{t-s} \Big(  \int p_{s}(x,z) (P_r \phi(z))^p  dz  \Big)^{1/p}  dr \cdot \Big(\int p_{s}(x,w)  V_{t-s}^{\phi,\phi}(w,w)^p dw \Big)^{1/p}.
\end{align*}
In the first inequality, we have used H\"older's inequality with some $p\in (1, \frac{10}{9})$ and $q>1$ such that $1/p+1/q=1$. We will pick $p=p(d, \alpha)$ close $1$ below. The last inequality uses Lemma \ref{l1.1}. By using \eqref{ae7.51}, we get
 \begin{align*} 
\int_0^{t-s} \Big(  \int p_{s}(x,z) [P_r \phi(z)]^p  dz  \Big)^{1/p}  dr &\leq \int_0^{t-s} \Big(  \int_{|z|\leq K} p_{s+r}(x,z)    dz   \Big)^{1/p}  dr\\
&\leq C(|x|^{2-\frac{d}{p}} \wedge 1),
 \end{align*}
 where the last inequality follows by Corollary \ref{ac4.6}. 
 On the other hand, again by Corollary \ref{ac4.6}, we have 
 \[
 V_{t-s}^{\phi,\phi}(w,w)^p \leq C\Big[|w|^{4p-2d}\wedge 1\Big], \quad \forall s,t\geq 0, w\in \R^d.
 \]
 We conclude from the above that 
 \begin{align} \label{ae8.23}
F_1(t,x)&\leq C \Big(|x|^{2-\frac{d}{p}}\wedge 1\Big)  \int_0^t \Big(\int p_{s}(x,w)  \Big[|w|^{4p-2d}\wedge 1\Big] dw \Big)^{1/p}  ds.
\end{align}
 
 Define
 \begin{align} \label{ae9.10}
 J_t(x):=\int_0^t \Big(\int p_{s}(x,y)  \Big[|y|^{4p-2d}\wedge 1\Big] dy \Big)^{1/p}  ds.
\end{align}
\begin{lemma}\label{al6.1}
Let $d\geq 5$ and $p\in (1, \frac{10}{9})$. For any $t\geq 1$ and $x\in \R^d$, we have
 \begin{align} \label{ae8.22}
J_t(x) \leq   C I_t(x),
 \end{align}
 where $I_t(x)$ are as in \eqref{ae9.11}.
\end{lemma}

\begin{proof}
First consider $|x|\leq 4$. Note that 
 \begin{align*} 
 \int_0^1 \Big(\int p_{s}(x,y)   \Big[|y|^{4p-2d}\wedge 1\Big]  dy \Big)^{1/p}  ds\leq \int_0^1 1^{1/p}  ds=1.
\end{align*}
Turning to $s\geq 1$, we have
 \begin{align*} 
 &\int p_{s}(x,y)   \Big[|y|^{4p-2d}\wedge 1\Big]  dy\leq  \int_{|y|\leq 1} \frac{1}{s^{d/2}}  dy+\frac{1}{s^{d/2}}  \int_{|y|> 1}  |y|^{4p-2d}  dy.
 \end{align*}
Since $p<\frac{10}{9} \leq \frac{d}{4}$, we get $4p-2d+d-1<-1$ and hence
 \begin{align} \label{ae8.2}
\int_{|y|> 1}  |y|^{4p-2d}  dy=C \int_1^\infty r^{4p-2d} r^{d-1} dr \leq C.
\end{align}
 It follows that
  \begin{align*} 
 &\int p_{s}(x,y)   \Big[|y|^{4p-2d}\wedge 1\Big]  dy \leq C\frac{1}{s^{d/2}},
\end{align*}
thus giving
 \begin{align} \label{ae9.39}
J_t(x) \leq 1+\int_1^\infty C s^{-\frac{d}{2p}} ds\leq C, \quad \forall |x|\leq 4.
\end{align}

Next, for any $|x|>4$, we have 
 \begin{align*} 
J_t(x)&\leq \int_0^t \Big(\int_{|y-x|\geq \frac{|x|}{2}} p_{s}(x,y)   \Big[|y|^{4p-2d}\wedge 1\Big]  dy \Big)^{1/p}  ds\\
 &+\int_0^\infty \Big(\int_{|y-x|<\frac{|x|}{2}} p_{s}(x,y)   \Big[|y|^{4p-2d}\wedge 1\Big]  dy \Big)^{1/p}  ds:= I_1+ I_2.
\end{align*}
For $I_1$, we use $|y-x|\geq \frac{|x|}{2}$ to get
\[
p_s(x,y) \leq Cs^{-\frac{d}{2}} e^{-\frac{|x|^2}{8s}},
\]
and so
 \begin{align} \label{ae9.21}
 I_1 &\leq \int_0^t \Big(Cs^{-\frac{d}{2}} e^{-\frac{|x|^2}{8s}} \int   \Big[|y|^{4p-2d}\wedge 1\Big]  dy \Big)^{1/p}  ds\nn\\
 &\leq C\int_0^t s^{-\frac{d}{2p}} e^{-\frac{|x|^2}{16s}} ds.
 \end{align}
The last inequality above follows by \eqref{ae8.2} and $p<2$.\\
 
 For $I_2$, we use $|y|\geq |x|-|y-x|\geq \frac{|x|}{2}>1$ to get
 \begin{align} \label{ae8.21}
 I_2 &\leq |x|^{4-\frac{2d}{p}} \int_0^\infty \Big(\int_{|y-x|<\frac{|x|}{2}} p_{s}(x,y)   dy \Big)^{1/p}  ds.
\end{align}
Notice that
 \begin{align*} 
&\int_{|y-x|<\frac{|x|}{2}} p_{s}(x,y)   dy  \leq C s^{-d/2} \int_{|y-x|<\frac{|x|}{2}}  e^{-\frac{|y-x|^2}{2s}} dy\\
&\leq C s^{-d/2} \int_{0}^{\frac{|x|}{2}}  e^{-\frac{r^2}{2s}} r^{d-1} dr=C   \int_{0}^{\frac{|x|^2}{8s}}  e^{-r} r^{\frac{d-2}{2}} dr.
\end{align*}
Therefore \eqref{ae8.21} becomes
 \begin{align*} 
 I_2 &\leq C|x|^{4-\frac{2d}{p}} \int_0^\infty \Big(\int_{0}^{\frac{|x|^2}{8s}}  e^{-r} r^{\frac{d-2}{2}} dr\Big)^{1/p}  ds \\
 &\leq C|x|^{6-\frac{2d}{p}} \int_0^\infty \frac{1}{u^2} \Big(\int_{0}^{u}  e^{-r} r^{\frac{d-2}{2}} dr\Big)^{1/p}  du.
\end{align*}
When $u\leq 1$, we have 
\begin{align*} 
\int_{0}^{u}  e^{-r} r^{\frac{d-2}{2}} dr\leq Cu^{\frac{d}{2}}.
\end{align*}
The integral of $0\leq u\leq 1$ is then bounded by
\begin{align*} 
\int_0^1 \frac{1}{u^2} Cu^{\frac{d}{2}} du \leq C.
\end{align*}
For $u\geq 1$, we simply use 
\begin{align*} 
\int_{0}^{u}  e^{-r} r^{\frac{d-2}{2}} dr\leq \int_{0}^{\infty}  e^{-r} r^{\frac{d-2}{2}} dr\leq C
\end{align*}
to get  the integral of $u\geq 1$ is at most
\begin{align*} 
\int_1^\infty  C\frac{1}{u^2}    du \leq C.
\end{align*}
 We conclude that $I_2  \leq  C|x|^{6-\frac{2d}{p}}$. One can check that when $t\geq 1$, we have
\begin{align*} 
\int_0^t s^{-\frac{d}{2p}} e^{-\frac{|x|^2}{16s}} ds &\geq \int_0^1 s^{-\frac{d}{2p}} e^{-\frac{|x|^2}{16s}} ds=C |x|^{2-\frac{d}{p}} \int_{\frac{|x|^2}{16}}^\infty r^{-2+\frac{d}{2p}} e^{-r} dr \\
&\geq C |x|^{2-\frac{d}{p}} \int_{1}^\infty r^{-2+\frac{d}{2p}} e^{-r} dr=C|x|^{2-\frac{d}{p}}\geq C|x|^{6-\frac{2d}{p}},
\end{align*}
where the last inequality follows by $|x|>4$ and $p<\frac{d}{4}$. The above implies that
 \begin{align*} 
I_2\leq \int_0^t s^{-\frac{d}{2p}} e^{-\frac{|x|^2}{16s}} ds.
\end{align*}

Together with \eqref{ae9.21}, we get
 \begin{align*} 
J_t(x)\leq C\int_0^t s^{-\frac{d}{2p}} e^{-\frac{|x|^2}{16s}} ds\leq C\int_0^t s^{\frac{d}{2}-\frac{d}{2p}} p_{8s}(x) ds, \quad \forall |x|>4.
\end{align*}
Recall from \eqref{ae9.39} that  $J_t(x)\leq C$ for all $|x|\leq 4$. Recalling that \eqref{ae8.25} gives
 \begin{align*} 
 \int_0^t s^{\frac{d}{2}-\frac{d}{2p}} p_{8s}(x) ds \leq C|x|^{2-\frac{d}{p}} <1 \quad \text{ when $|x|>1$ is large},
\end{align*}
we may complete the proof by adjusting the constant.
\end{proof}

 Apply the above lemma in \eqref{ae8.23} to see that  
 \begin{align} \label{ae8.24}
F_1(t,x)&\leq C  (|x|^{4-\frac{2d}{p}} \wedge 1),
\end{align}
  We conclude that
 \begin{align} \label{ae8.33}
\int  F_1(t,x)  dx\leq C+C\int_{|x|>1} |x|^{4-\frac{2d}{p}} dx \leq C,
\end{align}
where the last inequality follows by \eqref{ae8.59}. So the proof of \eqref{ae7.39} is complete.\\

\no {\bf Step 2.}  Now we deal with $F_2(t,x)$ from \eqref{ae8.28}. For any $x\in \R^d$ and $t\geq s\geq 0$, we have
   \begin{align} \label{ae8.32}
&\Pi_{(x,x)} \Big[    \phi(B_s)H({t-s},\widetilde{B}_s) e^{\int_0^s g(B_r, \widetilde{B}_r) dr}\Big]\nn\\
&=\int_0^{t-s}     \Pi_{(x,x)} \Big[   \phi(B_s) \Big( \int    p_{u}(\widetilde{B}_s, z)V_{t-s-u}^{\phi,\phi}(z,z) dz\Big)   e^{\int_0^s g(B_r, \widetilde{B}_r) dr}\Big] du.
 \end{align}
 Apply H\"older's inequality to see that the expectation above is at most
 \begin{align} \label{ae8.31}
&  \Big(\Pi_{(x,x)} \Big[  \phi(B_s)^p \Big( \int   p_{u}(\widetilde{B}_s, z)  V_{t-s-u}^{\phi,\phi}(z,z) dz \Big)^p  \Big]\Big)^{1/p} \Big(\Pi_{(x,x)} \Big[  e^{q\int_0^s g(B_r, \widetilde{B}_r) dr}\Big]\Big)^{1/q} \nn \\
&\leq C \Big(\int p_s(x,z)  \phi(z)^p dz\Big)^{1/p}  \Big(\Pi_{(x,x)} \Big[  \Big( \int   p_{u}(\widetilde{B}_s, z)  V_{t-s-u}^{\phi,\phi}(z,z) dz \Big)^p \Big]\Big)^{1/p},
 \end{align}
 where the last inequality follows by Lemma \ref{l1.1}.
 Using Jensen's inequality to see the second term in \eqref{ae8.31} is bounded by
  \begin{align*} 
&   \Big(\Pi_{(x,x)} \Big[ \int   p_{u}(\widetilde{B}_s, z)  \Big(V_{t-s-u}^{\phi,\phi}(z,z)\Big)^p dz  \Big]\Big)^{1/p} \\
&=    \Big(  \int     p_{s+u}(x, z)   \Big(V_{t-s-u}^{\phi,\phi}(z,z)\Big)^p dz  \Big)^{1/p}\leq  \Big(  \int     p_{s+u}(x, z)   (|z|^{4p-2d} \wedge 1) dz  \Big)^{1/p},
 \end{align*}
 where the last inequality uses Corollary \ref{ac4.6}.
 
 Returning to \eqref{ae8.32}, we get 
   \begin{align*}  
&\Pi_{(x,x)} \Big[    \phi(B_s)H({t-s},\widetilde{B}_s) e^{\int_0^s g(B_r, \widetilde{B}_r) dr}\Big]\\
&\leq C  \Big(\int p_s(x,z)  \phi(z)^p dz\Big)^{1/p}\int_0^{t-s}      \Big(  \int     p_{s+u}(x, z)   (|z|^{4p-2d} \wedge 1) dz  \Big)^{1/p} du\\
&\leq C \Big(\int_{|z|\leq K} p_s(x,z)  dz\Big)^{1/p} \cdot J_t(x).
 \end{align*}
 Therefore we have
 \begin{align*} 
& F_2(t,x)\leq C \int_0^t \Big(\int_{|z|\leq K} p_s(x,z)  dz\Big)^{1/p} ds \cdot J_t(x) \leq C(|x|^{2-\frac{d}{p}} \wedge 1)  \cdot J_t(x),
 \end{align*}
 where the last inequality uses Lemma \ref{la4.1} and \eqref{ae8.25}. The above gives the same bound as in \eqref{ae8.23} for $F_1(t,x)$.
 Hence, \eqref{ae8.27} follows similarly.

\subsection{Convergence of $I_3(T,\phi)$}
 
Recall $I_3(T, \phi)$ from \eqref{ae7.32} to see that 
\begin{align}\label{ae9.0}
\E[I_3(T, \phi)^2]=T^{-1}\int_0^T dt \int \int \E[v_T^\phi(t, x) v_T^\phi(t, y)] g(x,y) dx dy.
\end{align}
By using \eqref{ae7.31} and \eqref{ae7.34}, we get that for each $t\geq 0$ and $x,y\in \R^d$, 
\begin{align*}
&0\leq   v_T^\phi(t, x) v_T^\phi(t, y)  \leq   \Big[\frac{1}{2T^{1/2}}  \int_0^t ds \int p_{t-s}(x,z) V_1^\phi(s,z)^2 dz\nn\\
 &+\int_0^t ds \int p_{t-s}(x,z) v_T^\phi(s,z) \dot{W}(s,z) dz\Big]  \cdot \Big[\frac{1}{2T^{1/2}}  \int_0^t du \int p_{t-u}(y,w) V_1^\phi(u,w)^2 dw\nn\\
 &+\int_0^t du \int p_{t-u}(y,w) v_T^\phi(u,w) \dot{W}(u,w) dw\Big].
\end{align*}
Take expectation to see that
\begin{align*}
&  \E[ v_T^\phi(t, x) v_T^\phi(t, y)]  \leq   \frac{1}{T}  \int_0^t ds \int p_{t-s}(x,z) \int_0^t du \int p_{t-u}(y,w)  \E[V_1^\phi(s,z)^2 V_1^\phi(u,w)^2] dz dw\nn\\
 &+ \int_0^t ds \int \int p_{t-s}(x,z) p_{t-s}(y,w) \E[v_T(s,z)  v_T(s,w)] g(z,w) dz dw
\end{align*}

 We need to bound $\E[V_1^\phi(s,z)^2 V_1^\phi(u,w)^2]$.
Lemma \ref{al4.6} implies that 
\begin{align*}
\E[V_1^\phi(t,x) V_1^\phi(t,y)]\leq C\cdot   \widetilde{Q}(t,x)   \widetilde{Q}(t,y).
\end{align*}
 The following fourth-moment bounds follow similarly to the derivation of Lemma \ref{al4.6}.

\begin{lemma}\label{al6.3}
For any $t\geq 0$, $x_i\in \R^d$ for $1\leq i\leq 4$, we have
\begin{align*}
\E\Big[V_1^\phi(t,x_1) V_1^\phi(t,x_2)V_1^\phi(t,x_3)V_1^\phi(t,x_4)\Big]\leq C   \prod_{i=1}^4 \widetilde{Q}(t,x_i).
\end{align*}
\end{lemma}

\begin{proof}
The proof will be given in Appendix \ref{ac}.
\end{proof}

By using the above, we get
\begin{align*}
  \E[V_1^\phi(s,z)^2 V_1^\phi(u,w)^2]  &\leq  \Big(\E[V_1^\phi(s,z)^4]\Big)^{1/2} \Big(\E[ V_1^\phi(u,w)^4] \Big)^{1/2}\\
&\leq C\cdot   \widetilde{Q}(s, z)^2 \widetilde{Q}(u, w)^2.
 \end{align*}
 Therefore
\begin{align}\label{ae9.1}
&  \E[ v_T^\phi(t, x) v_T^\phi(t, y)]  \leq   \frac{1}{T}  \int_0^t ds \int p_{t-s}(x,z) \widetilde{Q}(s, z)^2 dz \int_0^t du \int p_{t-u}(y,w) \widetilde{Q}(u, w)^2 dw \nn\\
 &+ \int_0^t ds \int \int p_{t-s}(x,z) p_{t-s}(y,w) \E[v_T(s,z)  v_T(s,w)] g(z,w) dz dw.
\end{align}
Set
\begin{align}\label{ae9.4}
  H_0(t,x)=  \int_0^t ds \int p_{t-s}(x,z) \widetilde{Q}(s, z)^2 dz.
 \end{align}
Let   $F_0(t, x, y)$ be the solution of 
\begin{align}\label{ae9.2}
 F_0(t, x, y)  =   & T^{-1}   H_0(t,x)  H_0(t,y) \nn\\
 &+ \int_0^t ds \int \int p_{t-s}(x,z)   p_{t-s}(y,w)  F_0(s, z, w) g(z,w) dz dw.
\end{align}
By comparing \eqref{ae9.1} and \eqref{ae9.2}, we may apply Lemma \ref{l4.3} to obtain  
\begin{align*} 
&  \E[ v_T^\phi(t, x) v_T^\phi(t, y)]  \leq   F_0(t, x, y), \quad \forall t\geq 0, x,y\in \R^d.
\end{align*}
Using \eqref{ae9.0}  and the above, we may reduce the proof of \eqref{ea6.5}  to proving that
\begin{align}\label{ae9.3}
\lim_{T\to \infty} T^{-1}\int_0^T dt \int \int F_0(t, x, y) g(x,y) dx dy=0.
\end{align}

 Similar to the derivation of \eqref{ea9.99}, one can check that  
 \begin{align*} 
 \frac{\partial}{\partial t} F_0(t, x, y)&= \frac{\Delta_{(x,y)}}{2} F_0(t, x, y)  +F_0(t, x, y) g(x,y)\\
  &+  T^{-1}   H_0(t,x) \widetilde{Q}(t, y)^2 +T^{-1}   H_0(t,y)  \widetilde{Q}(t, x)^2.
\end{align*}
By Feynman-Kac's formula, we get
 \begin{align*} 
  F_0(t, x, y)&=T^{-1}\int_0^t   \Pi_{(x,y)} \Big[  H_0(t-s,B_s)   \widetilde{Q}(t-s, \widetilde{B}_s)^2  e^{\int_0^s g(B_r, \widetilde{B}_r) dr}\\
 &\quad +   H_0(t-s,\widetilde{B}_s)   \widetilde{Q}(t-s, {B}_s)^2 e^{\int_0^s g(B_u, \widetilde{B}_u) du} \Big] ds
\end{align*}
Define
\begin{align}\label{ae9.9}
  F_1(t, x, y)&:= \int_0^t   \Pi_{(x,y)} \Big[  H_0(t-s,B_s)   \widetilde{Q}(t-s, \widetilde{B}_s)^2  e^{\int_0^s g(B_r, \widetilde{B}_r) dr}  \Big] ds.
\end{align}
By symmetry, the proof of \eqref{ae9.3} can be reduced to proving that
\begin{align}\label{ae9.8}
\lim_{T\to \infty} T^{-2}\int_0^T dt \int \int F_1(t, x, y) g(x,y) dx dy=0,
\end{align}

Using \eqref{ae9.4}, we get
\begin{align}\label{ae9.6}
& \Pi_{(x,y)} \Big[  H_0(t-s,B_s)   \widetilde{Q}(t-s, \widetilde{B}_s)^2  e^{\int_0^s g(B_r, \widetilde{B}_r) dr}\Big]\nn\\
&=  \int_0^{t-s}    \Pi_{(x,y)} \Big[   \Big( \int   p_{t-s-u}({B}_s, z)  \widetilde{Q}(u,z)^2dz\Big)  \widetilde{Q}(t-s, \widetilde{B}_s)^2 e^{\int_0^s g(B_r, \widetilde{B}_r) dr}\Big] du\nn \\
&\leq   \int_0^{t-s} du   \Big(\Pi_{(x,y)} \Big[  \Big( \int   p_{t-s-u}({B}_s, z)  \widetilde{Q}(u,z)^2 dz\Big)^p  \widetilde{Q}(t-s, \widetilde{B}_s)^{2p} \Big]\Big)^{1/p}\nn \\
&\quad  \quad \quad \times \Big(\Pi_{(x,y)} \Big[  e^{q\int_0^s g(B_r, \widetilde{B}_r) dr}\Big]\Big)^{1/q}.
 \end{align}
 By Lemma \ref{l1.1}, we get
 \[
 \Big(\Pi_{(x,y)} \Big[  e^{q\int_0^s g(B_r, \widetilde{B}_r) dr}\Big]\Big)^{1/q} \leq 2^{1/q}.
 \]
 By Jensen's inequality, we have
 \[
  \Big( \int   p_{t-s-u}({B}_s, z)  \widetilde{Q}(u,z)^2 dz\Big)^p \leq    \int   p_{t-s-u}({B}_s, z)  \widetilde{Q}(u,z)^{2p} dz.
 \]
 Now conclude that \eqref{ae9.6} becomes
 \begin{align}\label{ae9.7}
& \Pi_{(x,y)} \Big[  H_0(t-s,B_s)   \widetilde{Q}(t-s, \widetilde{B}_s)^2  e^{\int_0^s g(B_r, \widetilde{B}_r) dr}\Big]\nn\\
 &\leq C  \int_0^{t-s}     \Big(\Pi_{(x,y)} \Big[  \int   p_{t-s-u}({B}_s, z)  \widetilde{Q}(u,z)^{2p} dz \cdot  \widetilde{Q}(t-s, \widetilde{B}_s)^{2p} \Big]\Big)^{1/p} du \nn\\
 &=C  \int_0^{t-s}     \Big(  \int   p_{t-u}(x, z)  \widetilde{Q}(u,z)^{2p} dz \cdot  \int p_s(y,w) \widetilde{Q}(t-s, w)^{2p} dw \Big)^{1/p} du\nn\\
 &\leq   C  \int_0^{t}     \Big(  \int   p_{u}(x, z)  \widetilde{Q}(t-u,z)^{2p} dz\Big)^{1/p} du \cdot  \Big( \int p_s(y,w) \widetilde{Q}(t-s, w)^{2p} dw \Big)^{1/p}.
 \end{align}
 Returning to \eqref{ae9.9}, we get
 \begin{align*} 
  F_1(t, x, y) \leq  C & \int_0^{t}     \Big(  \int   p_{u}(x, z)  \widetilde{Q}(t-u,z)^{2p} dz\Big)^{1/p} du \\
  &\times  \int_0^t \Big( \int p_s(y,w) \widetilde{Q}(t-s, w)^{2p} dw \Big)^{1/p} ds.
\end{align*}
 Use Lemma \ref{la4.1} and Corollary \ref{ac4.6} to bound $ \widetilde{Q}(t-u,z)^{2p}$ so that the above becomes
  \begin{align} 
  F_1(t, x, y)& \leq  C  J_t(x) J_t(y),
\end{align}
where $J_t(x)$ are as in \eqref{ae9.10}. So the proof of \eqref{ae9.8} is immediate if we show that
\begin{align} \label{ae10.1}
\lim_{T\to \infty} T^{-2}\int_0^T dt \int \int J_t(x) J_t(y) g(x,y) dx dy=0.
\end{align}
 Lemma \ref{al6.1} gives that $J_t(x)\leq CI_t(x)$ if $t\geq 1$. Recall from  Lemma \ref{al5.1} that if we let $\gamma=2+\delta$ for some $\delta\in (0, \frac{\alpha-2}{2}\wedge \frac{1}{2})$, then
  \begin{align*}   
  \int_{\R^d}  \int_{\R^d}  I_t(x) I_t(y) g(x,y) dx dy \leq Ct^{2+d-\frac{d}{p}-\frac{2+\delta}{2}}, \quad \forall t\geq 1.
 \end{align*}  
Pick $p \in (1, \frac{10}{9} \wedge \frac{d}{d-\frac{\delta}{2}})$ so that
\begin{align} \label{ae10.2}
d-\frac{d}{p}-\frac{\delta}{2}<0.
\end{align}
Using the monotonicity of $J_t(x) J_t(y)$ in $t$, we get that  for any $T>1$,
 \begin{align*}
 \int_0^T dt \int \int J_t(x) J_t(y) g(x,y) dx dy \leq C \int_0^T (t\vee 1)^{1+d-\frac{d}{p}-\frac{\delta}{2}}dt \leq CT^{2+d-\frac{d}{p}-\frac{\delta}{2}}.
\end{align*}
Hence \eqref{ae10.1} follows in view of \eqref{ae10.2}. The proof is now complete.

\section{Proof of Proposition \ref{p1.1}}\label{as7}

Recall from \eqref{ea5.1} that
\begin{align} \label{e2.5.0}
\la \lambda, V_1^\phi(T) \ra=T \la \lambda, \phi \ra+\int_0^T \int V_1^\phi(t,y) W(dt, y) dy.
\end{align}
Define 
\begin{align*} 
N^\phi(t):=\int_0^t \int V_1^\phi(s,y) W(ds, y) dy, \quad \forall t\geq 0.
\end{align*}
 Then $\{N^\phi(t), t\geq 0\}$ is an $\cF_t^W$ martingale with quadratic variation given by
\begin{align*}  
\la N^\phi\ra_t= \int_0^t ds \int \int V_1^\phi(s,x) V_1^\phi(s,y) g(x,y) dx dy.
\end{align*}
By using the monotonicity of $s\mapsto V_1^\phi(s,x) V_1^\phi(s,y)$, one immediately gets that $\P$-a.s.,
\begin{align} \label{e2.5.1}
\lim_{t\to \infty} t^{-1}\la N^\phi\ra_t= \lim_{t\to \infty}  \int \int  V_1^\phi(t,x) V_1^\phi(t,y) g(x,y) dx dy=\xi(W,\phi).
\end{align}
The following lemma implies that the limit above is a.s. finite.
\begin{lemma}\label{l7.1}
If $d\geq 5$ and $\alpha>4$, we have
\[
\E[\xi(W,\phi)]<\infty.
\]
\end{lemma}

\begin{proof}
Use Lemma \ref{ac4.6} to see that for any $t\geq 1$,
\begin{align*} 
\E\Big( \int \int  V_1^\phi(t,x) V_1^\phi(t,y) g(x,y) dx dy\Big) \leq C \int \int   I_t(x) I_t(y) g(x,y) dx dy \leq C,
\end{align*}
where the last inequality follows by Lemma \ref{al5.1} (b) as we are now in the scenario of the case $d\wedge \alpha>4$. Therefore, Fatou's Lemma gives that
\begin{align*} 
\E[\xi(W,\phi)] \leq \liminf_{t\to \infty} \E\Big( \int \int  V_1^\phi(t,x) V_1^\phi(t,y) g(x,y) dx dy\Big)  \leq C,
\end{align*}
as required.
\end{proof}

   By using the Dubins-Schwarz theorem (see, e.g., Revuz and Yor \cite{RY94}, Theorem V1.6 and V1.7), with an enlargement of the underlying probability space, we can construct some linear Brownian motion $(B_t, t\geq 0)$  such that
\begin{align}  
N^\phi(t)=B_{\la N^\phi\ra_t}.
\end{align}
It is clear that 
\begin{align}   \label{e2.5.2}
t^{-1/2}  N^\phi(t)=t^{-1/2}  B_{\la N^\phi\ra_t}\overset{d}{=} B_{t^{-1}\la N^\phi\ra_t}.
\end{align}
Now we are ready to give the proof of Proposition \ref{p1.1}.

\begin{proof}[Proof of Proposition \ref{p1.1}]

Combine \eqref{e2.5.0} and \eqref{e2.5.2} to see that it suffices to show that for any bounded and uniformly continuous function $h$, we have
\begin{align}  \label{ae4.21}
\lim_{t\to \infty} \E\Big[h\Big(B_{t^{-1}\la N^\phi\ra_t}\Big)\Big] = \E\Big[h\Big(B_{\xi(W,\phi)}\Big)\Big].
\end{align}

For any $\eps>0$, we choose $\delta>0$ such that $|h(x)-h(y)|<\eps$ holds for any $x,y\in \R$ with $|x-y|<\delta$. Then 
\begin{align*}  
 &\E\Big[\Big| h\Big(B_{t^{-1}\la N^\phi\ra_t}\Big)-h\Big(B_{\xi(W,\phi)}\Big)\Big|\Big]
  \leq \eps+2\|h\|_\infty \P\Big(\Big| B_{t^{-1}\la N^\phi\ra_t}-B_{\xi(W,\phi)}\Big|>\delta\Big)
\end{align*}
For any $\gamma \in (0,1)$, we have
\begin{align*}  
 & \P\Big(\Big| B_{t^{-1}\la N^\phi\ra_t}-B_{\xi(W,\phi)}\Big|>\delta\Big)\leq \P(|t^{-1}\la N^\phi\ra_t- \xi(W,\phi)|> \gamma)\\
 &+\P\Big(\Big| B_{t^{-1}\la N^\phi\ra_t}-B_{\xi(W,\phi)}\Big|>\delta, \ \Big|t^{-1}\la N^\phi\ra_t-\xi(W,\phi)\Big|\leq \gamma\Big)=:I_1+I_2.
\end{align*}
By using the a.s. convergence from \eqref{e2.5.1}, we get that when $t>0$ large, $I_1<\eps$.

To bound $I_2$, we first apply Lemma \ref{l7.1} to see that there is some $M>2$ such that
\begin{align*}  
 & \P(\xi(W,\phi)>M)<\eps.
\end{align*}
Then we have
\begin{align*}  
 I_2&\leq  \eps+\P\Big(\Big| B_{t^{-1}\la N^\phi\ra_t}-B_{\xi(W,\phi)}\Big|>\delta, |t^{-1}\la N^\phi\ra_t-\xi(W,\phi)|\leq \gamma, \xi(W,\phi)\leq M\Big) \\
 &\leq \eps+ \P\Big(\sup_{\substack{|s-t|\leq \gamma, \\0\leq s,t\leq M+1}} |B_s-B_t|>\delta\Big).
\end{align*}
By L\'evy's modulus of continuity, if we pick $\gamma$ small enough, then
\[
\P\Big(\sup_{\substack{|s-t|\leq \gamma, \\0\leq s,t\leq M+1}} |B_s-B_t|>\delta\Big)\leq \eps.
\]
Now we conclude that
\begin{align*}  
 &\E\Big[\Big| h\Big(B_{t^{-1}\la N^\phi\ra_t}\Big)-h\Big(B_{\xi(W,\phi)}\Big)\Big|\Big]
  \leq \eps+2\|h\|_\infty  (3\eps).
\end{align*}
Since $\eps>0$ is arbitrary, we finish the proof of \eqref{ae4.21}.
\end{proof}

 \bibliographystyle{plain}

\begin{thebibliography}{10}



\bibitem{CRZ15}
Zhen-Qing Chen, Yan-Xia Ren, and Guohuan Zhao.
\newblock Persistence and local extinction for superprocesses in random environments.
\newblock {\em Unpublished notes. }



 \bibitem {CDP}
T. Cox, R. Durrett, and E. Perkins.
\newblock Rescaled voter models converge to super Brownian motion.
\newblock {\em Ann. Probab.} {\bf 28}, 185-234, (2000).


 \bibitem {CP}
T. Cox and E. Perkins.
\newblock Rescaled Lotka-Volterra models converge to super-Brownian motion.
\newblock {\em Ann. Probab.} {\bf 33}, 904-947, (2005).


 
 
 \bibitem{Da14}
G. Da Prato and J. Zabczyk.
\newblock Stochastic equations in infinite dimensions.
\newblock {\em Cambridge University Press}, volume 152,  (2014).
 


\bibitem{Daw75}
D. Dawson.
\newblock Stochastic evolution equations and related measure-valued processes.
\newblock {\em J. Multivariate Analysis}, {\bf 5}: 1--52, (1975).

 

 
 

\bibitem {DS}
E. Derbez and G. Slade. 
\newblock The scaling limit of lattice trees in high dimensions.
\newblock {\em Commun. Math. Phys.}, {\bf 193}, 69-104, (1998).

\bibitem{DP99}
R. Durrett and E. Perkins.
\newblock Rescaled contact processes converge to super-Brownian motion in two or more dimensions.
\newblock {\em Probab. Theory Related Fields}, {\bf 114}, 309--399, (1999).

\bibitem{EK86}
S. Ethier and T. Kurtz.
\newblock Markov Processes: Characterization and Convergence. 
\newblock {\em John Wiley}, New York, (1986).


 

%\bibitem{Fol99}
%G. Folland.
%\newblock Real analysis: Modern Techniques and Their Applications. Second edition. 
%\newblock {\em Pure and Applied Mathematics (New York).}
%\newblock A Wiley-Interscience Publication. John Wiley\&Sons, Inc., New York, (1999). %\MR{1681462}


\bibitem {HS10}
R. van der Hofstad and A. Sakai.
\newblock Convergence of the critical finite-range contact process to super-Brownian motion above the upper critical dimension: The higher-point functions.
\newblock {\em Electron. J. Probab.}, {\bf 15}: no. 27, 801--894, (2010). %MR3878134

 \bibitem{HS03}
R. van der Hofstad and G. Slade.  
\newblock   Convergence of critical oriented percolation to super-Brownian motion above $4+1$ dimensions.
\newblock{\em Ann. Inst. H. Poincar\'e  Probab. Statist.} {\bf 39}, 413--485, (2003).


\bibitem{Hol08}
M. Holmes.
\newblock Convergence of lattice trees to super-Brownian motion above the critical dimension.
\newblock {\em Electron. J. Probab.}, 13: 671--755, (2008). 

 
 
 
 \bibitem{Hong23}
J. Hong.
\newblock Rescaled SIR epidemic processes converge to super-Brownian motion in four or more dimensions.
\newblock {\em Math Arxiv}, 2309.08926, (2023). 


 
\bibitem {Isc86}
I. Iscoe.
\newblock A weighted occupation time for a class of measured-valued branching processes.
\newblock{\em Probability Theory and Related Fields}, Volume 71, Issue 1, pp 85-116,  (1986). %Citations \MR{0814663}


\bibitem {Isc86b}
I. Iscoe.
\newblock Ergodic theory and a local occupation time for measure-valued critical branching Brownian motion.
\newblock{\em Stochastics}, Volume 18, Issue 1, pp 197-243,  (1986). 


 


 \bibitem{KS88}
N. Konno and T. Shiga.
\newblock Stochastic partial differential equations for some measure-valued diffusions.
\newblock {\em Probab. Theory Relat. Fields}, 79: 201--225,  (1988).
 
 

 
 


 
\bibitem{M96}
L. Mytnik.
\newblock Superprocesses in random environments.
\newblock {\em Ann. Probab.}, 24, 1953--1978, (1996). %MR1415235


 


\bibitem{MX07}
L. Mytnik and J. Xiong.
\newblock Local extinction for superprocesses in random environments.
\newblock {\em Elec. J. Probab.}, 12, 1349-1378, (2007).



\bibitem{Par}
\'E. Pardoux.
\newblock Stochastic partial differential equations--an introduction.
\newblock {\em SpringerBriefs Math.}, Springer, Cham, (2021).


 
\bibitem {RY94}
\newblock D. Revuz and M. Yor.
\newblock Continuous Martingales and Brownian Motion, Springer, Berlin,  (1994).

 \bibitem{Sug89}
S. Sugitani.
\newblock Some properties for the measure-valued branching diffusion processes.
\newblock {\em J. Math. Soc. Japan}, {\bf 41}:437--462, (1989). %MR999507


 \bibitem{Wal86}
J. Walsh.
\newblock  An Introduction to Stochastic Partial Differential Equations.
\newblock {\em Lect. Notes. in Math., no.\ 1180, Ecole d'Et\'e de Probabilit\'es de Saint Flour 1984.}
\newblock Springer, Berlin (1986).

\bibitem{Wat68}
S. Watanabe.
\newblock A limit theorem of branching processes and continuous state branching processes.
\newblock {\em J. Math. Kyoto Univ.}, {\bf 8}: 141--167, (1968).
 
 
 \bibitem{WZ15}
R. Wheeden and A. Zygmund.
\newblock Measure and integral. An introduction to real analysis. Second edition.
\newblock {\em Pure Appl. Math. (Boca Raton)}, CRC Press, Boca Raton, FL, (2015).
 
  

 \bibitem{Ying}
J. Ying.
\newblock Dirichlet forms perturbed by additive functionals of extended Kato class.
\newblock {\em Osaka J. Math.}, 34, 933-952, (1997).

  

 \end{thebibliography}
\def\cprime{$'$}

 \appendix
 
 \section{Proof of Lemma \ref{l1.1}} \label{a1}

\begin{proof} [Proof of Lemma \ref{l1.1}]
 Fix any $\alpha>2$, $q>0$ and $d\geq 3$.
Define
\[
\bar{g}(z)=\eps (|z|^{-\alpha} \wedge 1), \quad \forall z\in \R^d,
\]
for some $\eps>0$.
By assumption \eqref{e0.0}, we have for any $t>0$ and $x,y\in \R^d$,
\begin{align}\label{ae6.71}
\Pi_{(x,y)} \Big(e^{q\int_0^t   g(B_s, \widetilde{B}_s) ds} \Big) \leq \Pi_{(x,y)} \Big(e^{q \int_0^t   \bar{g}(B_s-\widetilde{B}_s) ds} \Big)=E_{x-y} \Big(e^{q \int_0^t   \bar{g}(\beta_s) ds} \Big),
\end{align}
where $\beta=(\beta_t, t\geq 0)$ is a $d$-dimensional Brownian motion starting from $z\in \R^d$ under $E_z$ such that $Var(\beta_t)=2t$. 
One can check that $(q\int_0^t   \bar{g}(\beta_s)ds, {t\ge 0})$ is an additive functional of the Brownian motion $\beta$ (see Page 935 of \cite{Ying} for the precise definition). By applying Lemma 2.1 of \cite{Ying}, we get that for any $t>0$, if 
\begin{align}\label{ea4.5}
\sup_{z\in \R^d} E_{z} \Big( q \int_0^s   \bar{g}(\beta_s) ds \Big) \leq \frac{1}{2}, \quad \forall s<t,
\end{align}
then it follows that
\begin{align*}
\sup_{z\in \R^d} E_{z}  \Big(e^{q \int_0^t   \bar{g}(\beta_s) ds} \Big) \leq (1-\frac{1}{2})^{-1}=2.
\end{align*}

Once confirming \eqref{ea4.5}, we may easily conclude from \eqref{ae6.71} that for all $t>0$ and $x,y\in \R^d$, we have
\begin{align*} 
\Pi_{(x,y)} \Big(e^{q\int_0^t   g(B_s, \widetilde{B}_s) ds} \Big) \leq 2.
\end{align*}
The conclusion follows by letting $t\to \infty$ and taking supremum over all $x,y$. \\

It remains to prove \eqref{ea4.5}. To see this, we note that
\begin{align*} 
  E_{z} \Big( q \int_0^\infty \bar{g}(\beta_s) ds \Big)=q \int_{\R^d} \bar{g}(y) G(z,y) dy,
\end{align*}
where $G(z,y)$ is the Green function of $(\beta_t)$ given by
\begin{align*} 
 G(z,y)=\int_0^\infty p_{2s}(z,y) ds =C |z-y|^{2-d}.
\end{align*}
Therefore for any $z\in \R^d$, we have
\begin{align*} 
  &q \int_{\R^d} \bar{g}(y) G(z,y) dy= C \eps q  \int_{\R^d} (|y|^{-\alpha} \wedge 1) |z-y|^{2-d} dy\\
 &\leq C\eps q\int_{\R^d} (|y|^{-\alpha} \wedge 1) |y|^{2-d} dy \\
 &\leq  C \eps q \Big[C  +C \int_1^\infty r^{2-d-\alpha} r^{d-1} dr \Big]   \leq C(\alpha, d) \eps q,
\end{align*}  
where the first inequality follows by Lemma 3.6 of Sugitani \cite{Sug89}, and the last inequality follows by $\alpha>2$. 
By letting $\eps\leq  \frac{1}{2C(\alpha, d) q}$, we obtain that
\begin{align} 
\sup_{z\in \R^d}   E_{z} \Big( q \int_0^\infty \bar{g}(\beta_s) ds \Big) \leq \frac{1}{2},
\end{align}
as required.
\end{proof}

 \section{Generalized Gronwall's inequality} \label{a2}
 
This section will give the proof of Lemma \ref{l4.3}.
 
\begin{proof} [Proof of Lemma \ref{l4.3}]

 Define $H(t,x)= G(t,x)-F(t,x)$ for each $t\geq 0$ and $x\in \R^d$. By assumption we have for any $t\geq 0$ and $x\in {\R^d}$,
 \begin{align*}
 H(t,x)\leq   \int_0^t ds \int_{\R^d} p_{t-s}(x,y) H(s,y)dy.
 \end{align*}
 Set $H^+(t,x)=H(t,x) \vee 0$. Then the above implies that
   \begin{align*}
 H^+(t,x)\leq   \int_0^t ds \int_{\R^d} p_{t-s}(x,y) H^+(s,y)dy.
 \end{align*}
 Define $$h_t= \int_{\R^d} H^+(t,x) dx$$ so that
   \begin{align}\label{e5.44}
h_t &\leq   \int_{\R^d} \int_0^t ds \int_{\R^d} p_{t-s}(x,y) H^+(s,y)dy dx\nn\\
&= \int_0^t ds \int_{\R^d} H^+(s,y)dy=  \int_0^t h_s ds,
 \end{align}
 where the first equality follows by Fubini's theorem. Apply Gronwall's inequality (see, e.g., Theorem 5.1 in Appendix of Ethier-Kurtz \cite{EK86}) with \eqref{e5.44} to see that
 \begin{align*} 
h_t=0, \quad \forall t\geq 0.
 \end{align*}
 So for any $t\geq 0$, $H^+(t,x)=0$  for a.e. $x \in  {\R^d}$. By the continuity of $x\mapsto H^+(t,x)$, we get $H^+(t,x)=0$  for all $x \in  {\R^d}$, thus giving \eqref{e5.45}.
\end{proof}

\section{Fourth moment bounds} \label{ac}

In this section, we will prove Lemma \ref{al6.3}.
Recall from \eqref{ae10.33} that
 \begin{align} 
dV_1^{\phi}(t,x)=\phi(x)+  \frac{\Delta}{2} V_1^{\phi}(t,x) dt   +  V_1^{\phi}(t,x)   {W}(dt,x).
\end{align}
For any $x_1, x_2\in \R^d$, by Ito's formula, one can check that
\begin{align*} 
 & d(V_1^{\phi}(t, x_1)V_1^{\phi}(t, x_2)) =    \frac{\Delta_{(x_1,x_2)}}{2} V_1^{\phi}(t, x_1) V_1^{\phi}(t, x_2) dt\\
 &+g(x_1,x_2) V_1^{\phi}(t, x_1)V_1^{\phi}(t, x_2) dt+V_1^{\phi}(t, x_1) V_1^{\phi}(t, x_2) \Big[{W}(dt,x_1)+ {W}(dt,x_2)\Big]\\
 &+\Big[\phi(x_1) V_1^{\phi}(t, x_2) +\phi(x_2) V_1^{\phi}(t, x_1)\Big] dt.
 \end{align*}
Similarly, one can check by induction that for any $n\geq 1$ and any $x_1$, $\cdots$, $x_n \in \R^d$,
 \begin{align*} 
d\prod_{i=1}^n V_1^{\phi}(t,x_i)&= \frac{\Delta_{(x_1, \cdots, x_n)}}{2} \prod_{i=1}^n V_1^{\phi}(t,x_i) dt   +  \prod_{i=1}^n V_1^{\phi}(t,x_i) \Big[ \sum_{1\leq i< j \leq n} g(x_i, x_j)\Big] dt\\
&+ \prod_{i=1}^n V_1^{\phi}(t,x_i)  \Big[  \sum_{i=1}^n {W}(dt,x_i)\Big]+\sum_{k=1}^n \phi(x_k) \prod_{1\leq i\neq k \leq n} V_1^{\phi}(t,x_i) dt.
\end{align*}
 Take expectation to see that
  \begin{align*} 
&d \E\Big[\prod_{i=1}^n V_1^{\phi}(t,x_i)\Big]= \frac{\Delta_{(x_1, \cdots, x_n)}}{2} \E\Big[\prod_{i=1}^n V_1^{\phi}(t,x_i)\Big] dt\\
&+\E\Big[\prod_{i=1}^n V_1^{\phi}(t,x_i)\Big] \Big[ \sum_{1\leq i< j \leq n} g(x_i, x_j)\Big] dt    +\sum_{k=1}^n \phi(x_k) \E\Big[\prod_{1\leq i\neq k \leq n} V_1^{\phi}(t,x_i)\Big] dt.
\end{align*}
 Using Feymann-Kac's formula, we obtain that
   \begin{align} \label{ae10.31}
 \E\Big[\prod_{i=1}^n V_1^{\phi}(t,x_i)\Big]= \int_0^t \Pi_{(x_1, \cdots, x_n)} \Bigg\{ &\sum_{k=1}^n \phi(B_s^k) \E\Big[\prod_{1\leq i\neq k \leq n} V_1^{\phi}(t-s,B_s^i)\Big] \nn\\
&e^{\int_0^s\sum_{1\leq i< j \leq n} g(B_r^i, B_r^j) dr}\Bigg\}ds.
\end{align}
In the above, we let $\{B_s^k, s\geq 0\}_{1\leq k\leq n}$ be independent $d$-dimensional Brownian motion starting respectively from $(x_k, 1\leq k\leq n)$ under $\Pi_{(x_1, \cdots, x_n)}.$\\

We are ready to give the proof of Lemma \ref{al6.3}.
\begin{proof}[Proof of Lemma \ref{al6.3}]
Set $n=3$ in \eqref{ae10.31} to see that
  \begin{align} \label{ae9.14}
 \E\Big[\prod_{i=1}^3 V_1^{\phi}(t,x_i)\Big]= \int_0^t \sum_{k=1}^3 \Pi_{(x_1, x_2, x_3)} \Bigg\{ & \phi(B_s^k) \E\Big[\prod_{1\leq i\neq k \leq 3} V_1^{\phi}(t-s,B_s^i)\Big] \nn\\
&e^{\int_0^s\sum_{1\leq i< j \leq 3} g(B_r^i, B_r^j) dr}\Bigg\}ds.
\end{align}
Consider the case for $k=1$ above, and let
 \begin{align*} 
 I&:=\int_0^t ds\Pi_{(x_1, x_2, x_3)}  \Bigg\{  \phi(B_s^1) \E\Big[V_1^{\phi}(t-s,B_s^2)V_1^{\phi}(t-s,B_s^3)\Big] e^{\int_0^s\sum_{1\leq i< j \leq 3} g(B_r^i, B_r^j) dr}\Bigg\}.
 \end{align*}
Lemma \ref{l2.2} and Lemma \ref{al4.6} give that
 \begin{align*} 
 \E\Big[  V_1^{\phi}(t-s,B_s^2)V_1^{\phi}(t-s,B_s^3)\Big] \leq C  \cdot \widetilde{Q}(t-s,B_s^2)   \widetilde{Q}(t-s,B_s^3).
 \end{align*}
 Plug in the above and recall the definition of $\widetilde{Q}(t,x)$ from \eqref{ae9.41} to see that
  \begin{align} \label{ae9.13}
 I& \leq C\int_0^t ds \int_{0}^{t-s} dl \int_{0}^{t-s}  du \ \Pi_{(x_1, x_2, x_3)} \Bigg\{  \phi(B_s^1)  \Big[\int_{|z_2|\leq K} p_u(B_s^2, z_2) dz_2\Big]^{1/p}\nn \\
&\quad \quad\quad\quad\quad \Big[\int_{|z_3|\leq K} p_l(B_s^3, z_3) dz_3\Big]^{1/p} e^{\int_0^s\sum_{1\leq i< j \leq 3} g(B_r^i, B_r^j) dr}\Bigg\}\nn\\
&\leq C\int_0^t ds \int_{0}^{t-s} dl \int_{0}^{t-s}  du \ \Big(\Pi_{(x_1, x_2, x_3)}  \Big[e^{q\int_0^s\sum_{1\leq i< j \leq 3}   g(B_r^i, B_r^j) dr}\Big] \Big)^{1/q} \nn \\
&\quad \quad\quad\quad\quad \Big(\Pi_{(x_1, x_2, x_3)} \Big[ \phi(B_s^1)^p  \int_{|z_2|\leq K} p_u(B_s^2, z_2) dz_2 \int_{|z_3|\leq K} p_l(B_s^3, z_3) dz_3\Big]\Big)^{1/p}.
  \end{align}
  The last inequality above follows by  H\"older's inequality.\\
  
Apply a generalized H\"older's inequality (see, e.g., Exercise 6 of Chapter 8 in \cite{WZ15}) to see that 
 \begin{align*}
&  \Pi_{(x_1, x_2, x_3)}  \Big[e^{q\int_0^s\sum_{1\leq i< j \leq 3}  g(B_r^i, B_r^j) dr}\Big] 
\leq  \prod_{1\leq i< j \leq 3} \Big( \Pi_{(x_1, x_2, x_3)}  \Big[e^{3q\int_0^s  g(B_r^i, B_r^j) dr}\Big] \Big)^{1/3}\leq 2,
\end{align*}
where the last inequality follows from  Lemma \ref{l1.1} with $q=3q.$ Hence \eqref{ae9.13} becomes
  \begin{align*}  
 I& \leq C \int_0^t ds\int_{0}^{t-s} dl \int_{0}^{t-s}  du  \ \Big(\int_{|z_1|\leq K} p_s(x_1, z_1)   dz_1 \\
 &\quad \quad\quad\quad\quad \int_{|z|\leq K} p_{s+u}(x_2, z_2) dz_2 \int_{|z_3|\leq K} p_{s+l}(x_3, z_3) dz_3 \Big)^{1/p}\\
 &\leq C \cdot \widetilde{Q}(t, x_1)  \widetilde{Q}(t, x_2) \widetilde{Q}(t, x_3).
   \end{align*}
 Returning to \eqref{ae9.14}, by symmetry we may conclude that
  \begin{align} \label{ae9.15}
 \E\Big[\prod_{i=1}^3 V_1^{\phi}(t,x_i)\Big]\leq C \cdot \widetilde{Q}(t, x_1)  \widetilde{Q}(t, x_2) \widetilde{Q}(t, x_3).
\end{align}

Set $n=4$ in \eqref{ae10.31}  to see that
  \begin{align} \label{ae9.16}
 \E\Big[\prod_{i=1}^4 V_1^{\phi}(t,x_i)\Big]= \int_0^t \sum_{k=1}^4 \Pi_{(x_1, x_2, x_3)} \Bigg\{ & \phi(B_s^k) \E\Big[\prod_{1\leq i\neq k \leq 4} V_1^{\phi}(t-s,B_s^i)\Big] \nn\\
&e^{\int_0^s\sum_{1\leq i< j \leq 4} g(B_r^i, B_r^j) dr}\Bigg\}ds.
\end{align}
By using \eqref{ae9.15} and \eqref{ae9.16}, one may repeat the above arguments for $n=3$ to get that 
  \begin{align*}  
 \E\Big[\prod_{i=1}^4 V_1^{\phi}(t,x_i)\Big]\leq C \cdot \widetilde{Q}(t, x_1)  \widetilde{Q}(t, x_2) \widetilde{Q}(t, x_3) \widetilde{Q}(t, x_4).
\end{align*}
The proof is now complete.
  \end{proof}

\end{document}